%% file: interface.tex
\documentclass[reqno,11pt]{article}
\usepackage{graphicx, amsmath, amsthm}
\usepackage{amssymb}  
\usepackage{dsfont} 
\usepackage{stmaryrd}
\usepackage{gensymb}

\input{interface.sty}
\input{graph_inter.sty}

\begin{document}


\title{Interface dynamics of a metastable \\
mass-conserving spatially extended diffusion}
\author{Nils Berglund, S\'ebastien Dutercq}
\date{}   

\maketitle

\begin{abstract}
\noindent
We study the metastable dynamics of a discretised version of the mass-conserving
stochastic Allen--Cahn equation. Consider a periodic one-dimensional lattice
with $N$ sites, and attach to each site a real-valued variable, which can be
interpreted as a spin, as the concentration of one type of metal in an alloy, or
as a particle density. Each of these variables is subjected to a local force
deriving from a symmetric double-well potential, to a weak ferromagnetic
coupling with its nearest neighbours, and to independent white noise. In
addition, the dynamics is constrained to have constant total magnetisation or
mass. Using tools from the theory of metastable diffusion processes, we show
that the long-term dynamics of this system is similar to a Kawasaki-type
exchange dynamics, and determine explicit expressions for its transition
probabilities. This allows us to describe the system in terms of the dynamics of
its interfaces, and to compute an Eyring--Kramers formula for its spectral gap.
In particular, we obtain that the spectral gap scales like the inverse system
size squared.
\end{abstract}

\leftline{\small{\it Date.\/} August 18, 2015.}
\leftline{\small 2010 {\it Mathematical Subject Classification.\/} 
60J60,   
60K35    
(primary), 
82C21,   
82C24    
(secondary)
}
\noindent{\small{\it Keywords and phrases.\/}
Metastability,
Kramers' law, 
stochastic exit problem, 
Allen--Cahn equation, 
Kawasaki dynamics, 
interface,
spectral gap
.}  


\section{Introduction}
\label{sec_intro}

The low-temperature dynamics of spatially extended systems often displays
metastability: these systems can spend considerable amounts of time in
configurations that have higher energy than their ground state. Well-known
examples of such phenomena are supercooled water, which remains liquid at
temperatures below $0$\degree C, a supersaturated gas, which does not
condensate although this would be thermodynamically more favourable, and a
wrongly magnetised ferromagnet. 

Much research effort has been dedicated to the study of metastable lattice
systems, such as the Ising model at low temperature. This has led to very
precise results on the time the system spends in metastable equilibrium, on the
way it moves from a metastable to a stable state by creating a critical droplet,
and on the shape of this droplet. See for instance~\cite{denHollander04} for a
review on Ising models with Glauber (spin flip) dynamics and lattice gases with 
Kawasaki (particle/hole exchange) dynamics, and~\cite{OlivieriVares05} for
results based on the theory of large deviations. A considerably more difficult
case arises when there is no underlying lattice given a priori, but particles
instead evolve in $\R^d$, and one wants to describe processes such as
crystallisation. For recent results in this direction, see for
instance~\cite{Jansen_Jung_14,Flatley_Theil_15,denHollander_Jansen_15}. 

Another type of models whose metastable behaviour is understood in detail are
diffusion processes described by stochastic differential equations with weak
noise. A general large-deviation approach to these equations goes back to the
work of Freidlin and Wentzell~\cite{FW}, which provides many results on
transition times between attractors and on the long-time dynamics. In the case
of reversible diffusions (that is, those satisfying a detailed balance
condition), metastable timescales are governed by the so-called Eyring--Kramers
formula, derived heuristically in~\cite{Eyring,Kramers}, and first proved in a
mathematically rigorous way in~\cite{BEGK,BGK}. See for
instance~\cite{Berglund_irs_MPRF} for a recent survey on various methods of
proof and extensions of the result. 

A spatially extended system of coupled diffusions, which can be considered of
intermediate difficulty between lattice systems with discrete spins and systems
of particles evolving in $\R^d$, was introduced in~\cite{BFG06a,BFG06b}. In this
model, the spins are still attached to a lattice (which is periodic and
one-dimensional of size $N$), but they take values in $\R$ instead of
$\set{-1,+1}$. Each spin feels a local symmetric double-well potential with
minima in $\pm1$, and is coupled ferromagnetically to its nearest neighbours. In
addition, each spin is subjected to independent white noise. For weak coupling,
the dynamics of this system was shown to be similar to that of an Ising model
with Glauber spin-flip dynamics. Indeed, the energy of configurations increases
with the number of \emph{interfaces}, defined as pairs of neighbouring spins
having different sign. As a consequence, the system favours configurations with
few clusters of spins having the same sign. On the other hand, when the coupling
scales like $N^2$, the system converges as $N\to\infty$ to an Allen--Cahn SPDE
with space-time white noise, whose metastable behaviour was studied
in~\cite{BG12a,Barret_2015}. 

A natural question that arises is whether one can construct a similar system,
with continuous spins attached to a discrete lattice, but whose dynamics for
weak coupling resembles Kawasaki exchange dynamics instead of Glauber spin-flip
dynamics. In other words, one would like to impose that the total magnetisation
(or the total mass in lattice gas terminology) is conserved. A simple way of
doing this is to start with the potential energy of the system considered
in~\cite{BFG06a,BFG06b}, and to constrain it to the hypersurface where the sum
of all spins is constant, say equal to zero. This is nothing but the
discretised version of the mass-conserving Allen--Cahn equation introduced
in~\cite{Rubinstein_Sternberg_92}. The objective of the present work is to
study the metastable dynamics of this model. 

It is quite easy to see that in the uncoupled limit, the potential energy of the
constrained system is minimal when exactly half the sites have value $+1$, while
the other half have value~$-1$. Such states have a clear particle system
interpretation: just consider each $+1$ as a particle and each $-1$ as a hole.
As in the unconstrained case, for weak positive coupling, the energy of
configurations increases with the number of interfaces. Therefore the ground
state consists of the configurations having exactly one cluster of particles and
one cluster of holes, separated by two interfaces. Higher-energy configurations
have more clusters and more interfaces. Thus if the system starts in an excited
state with many interfaces, one expects that its clusters will gradually merge,
reducing the number of interfaces, until the ground state is reached
(\figref{fig:interface_dynamics}). 

While our analysis will show that this picture is essentially correct, there is
a complication due to the fact that particle/hole configurations are \emph{not}
the only local minima of the potential energy. Somewhat unexpectedly, there
turn out to be many more \lq\lq spurious\rq\rq\ local minima, whose coordinates
are not close to $\pm1$. The way around this difficulty is to realise that all
spurious configurations have a higher energy than the particle/hole
configurations. Therefore the long-term dynamics will spend most of the time
near the particle/hole configurations, with occasional transitions between
them. Our main result is the characterisation of this effective dynamics. 

\begin{figure}
\centerline{
\includegraphics*[clip=true,width=0.8\textwidth]{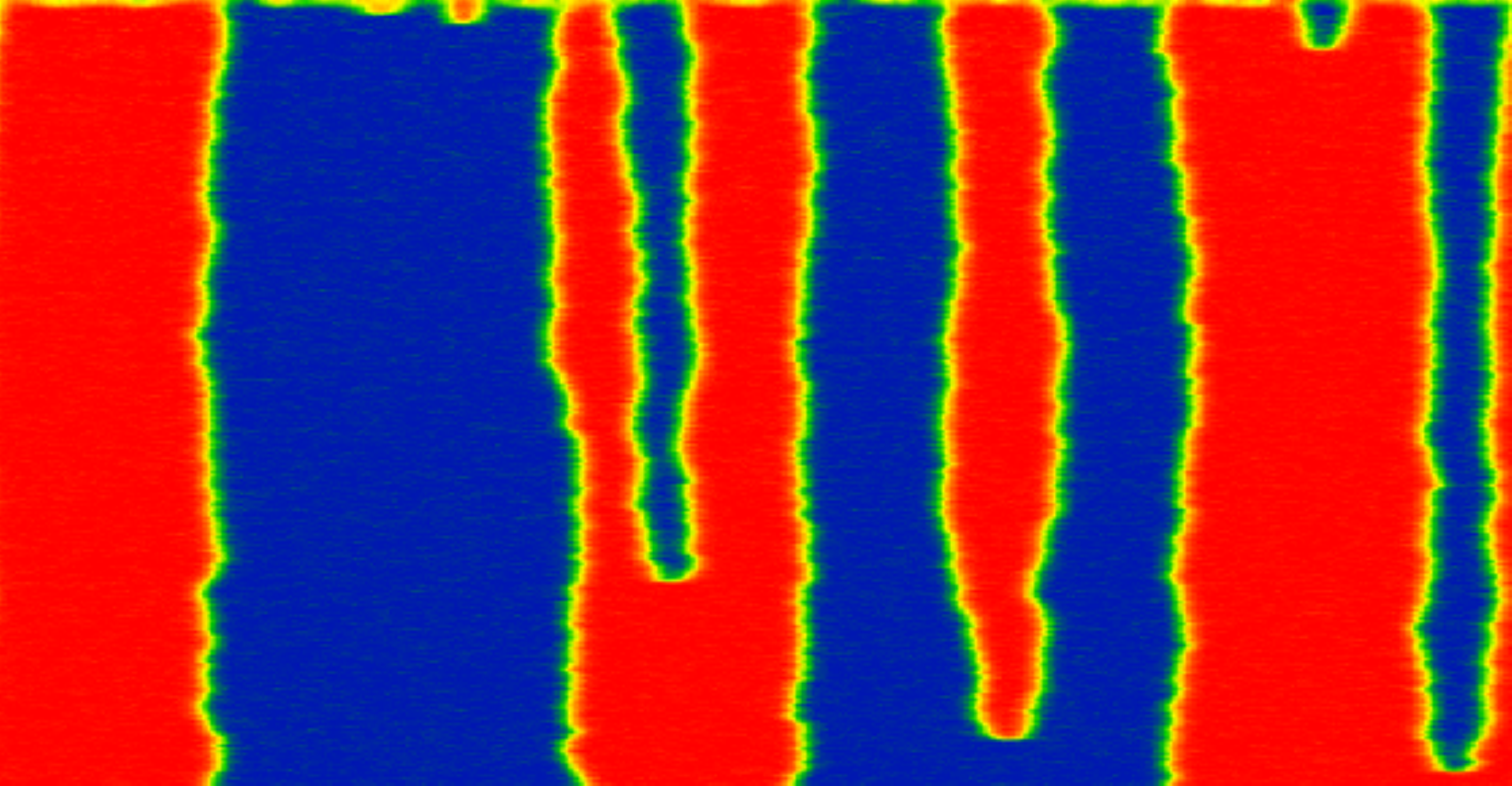} 
}
\caption[]{Example of evolution of the constrained system~\eqref{eq:model07}
with $N=512$ sites. Space goes from left to right, and time from top to bottom.
Blue and red correspond to spin values close to $-1$ and $1$ respectively. The
system starts in a configuration with $40$ interfaces, many of which disappear
quickly. At the end of the simulation, the number of interfaces has been reduced
to $4$. Parameter values are $\eps=0.02$ and $\gamma=16$. This coupling
intensity, which is much larger than considered in this work, has been chosen to
obtain transitions on an observable timescale.  
}
\label{fig:interface_dynamics}
\end{figure}

This paper is organised as follows. In Section~\ref{sec_model}, we give a
precise definition of the considered model. In Section~\ref{sec_landscape}, we
describe the potential landscape of the model, meaning that we find all local
minima of the potential energy, and describe how they are connected by saddles
with one unstable direction. Section~\ref{sec_meta} uses the notion of
metastable hierarchy to show that the dynamics indeed concentrates on
particle/hole configurations, and derives the effective dynamics on these
states. In Section~\ref{sec_dynamics} we use this information to characterise
the evolution of interfaces, and we derive a sharp estimate for the spectral
gap of the system, which determines the relaxation time to equilibrium. 
Section~\ref{sec_conclusion} contains concluding remarks, while most
proofs are postponed to the appendix. 

\medskip

\noindent
\textbf{Notations:} If $i\leqs j$ are integers, $\intint{i}{j}$ denotes the set
$\set{i,i+1, \dots, j}$. The cardinality of a finite set $A$ is denoted by
$\abs{A}$, and $A=B\cupdot C$ indicates that $A=B\cup C$ with $B$ and $C$
disjoint. We write $1_A$ for the indicator function of the set $A$, $\one_n$ or
simply $\one$ for the identity matrix of size $n\times n$, and $\vone$ for a
column vector with all components equal to $1$. Finally, we write
$\expecin{\mu}{\cdot}$ for expectations with respect to the law of the diffusion
process started with distribution $\mu$, and $\expecin{x}{\cdot}$ in case $\mu$
is concentrated in a single point $x$.

\medskip

\noindent
\textbf{Acknowledgements:} The idea of studying the constrained process
considered in this work goes back to a question Erwin Bolthausen asked after a
talk given in Z\"urich by the first author on the unconstrained model studied 
in~\cite{BFG06a,BFG06b}. 


\section{Definition of the model}
\label{sec_model}

Consider the potential $V_\gamma:\R^N\to\R$ defined by 
\begin{equation}
 \label{eq:model01}
 V_\gamma(x) = \sum_{i=1}^N U(x_i)
 + \frac{\gamma}{4} \sum_{i=1}^N (x_{i+1}-x_i)^2\;,
 \qquad
 U(\xi)=\frac14\xi^4 - \frac12\xi^2\;,
\end{equation} 
where $N\geqs2$ is an integer and $\gamma\geqs0$ is a coupling parameter. We
also make the identification $x_{N+1}=x_1$, that is, we consider periodic
boundary conditions. Thus $x$ can be considered either as an element of $\R^N$,
or as an element of $\R^\Lambda$, where $\Lambda$ is the periodic
lattice $\Z/N\Z$. 

The potential $V_\gamma$ allows to define a diffusion process by the stochastic
differential equation 
\begin{equation}
 \label{eq:model02}
 \6x_t = -\nabla V_\gamma(x_t)\6t + \sqrt{2\eps}\6W_t\;,
\end{equation} 
where $W_t$ is an $N$-dimensional Wiener process, and $\eps\geqs0$ is a small
parameter measuring noise intensity. The dynamics of this system has already
been studied in~\cite{BFG06a,BFG06b}. Here we are interested in a different
system, obtained by constraining the diffusion to the hyperplane 
\begin{equation}
 \label{eq:model03}
 S = \biggsetsuch{x\in\R^N}{\sum_{i=1}^N x_i=0}\;.
\end{equation} 
To define its dynamics, let $R$ be an orthogonal matrix mapping the unit
normal vector to $S$ to the $N$th canonical basis vector $e_N$. Let $\widehat
V_\gamma(y)=V_\gamma(R^{-1}y)$, and define the dynamics by 
\begin{align}
\nonumber
\6y_{i,t} &= -\frac{\partial\widehat V_\gamma(y)}{\partial y_{i,t}} \6t +
\sqrt{2\eps} \6W_{i,t}\;, 
 \qquad 
 i=1,\dots,{N-1}\;, \\
 y_{N,t} &= 0\;, 
 \label{eq:model04}
\end{align} 
where $W_{1,t},\dots,W_{N-1,t}$ are independent Brownian motions. Then $x_t$ is
by definition the process $x_t=R^{-1}y_t$. It is easy to check that this
definition does not depend on the choice of $R$.  

An equivalent way of defining the dynamics is to write 
\begin{equation}
 \label{eq:model05}
 \6y_t = \bigbrak{-\nabla \widehat V_\gamma(y_t) + \pscal{\nabla \widehat
V_\gamma(y_t)}{e_N}e_N} + \sqrt{2\eps} \6W_t
\end{equation} 
where $W_t$ is an $(N-1)$-dimensional Wiener process. Indeed, the extra term
precisely ensures that the $N$-th component of the drift term
vanishes. Transforming back, we obtain the equation 
\begin{equation}
 \label{eq:model06}
 \6x_t = \biggbrak{-\nabla V_\gamma(x_t) + \frac{1}{N} \pscal{\nabla
V_\gamma(x_t)}{\vone}\vone} \6t + \sqrt{2\eps}\6\widetilde W_t\;,
\end{equation} 
where $\vone$ denotes the vector with all components equal to $1$ (hence the
normalisation $1/N$), and $\widetilde W_t = R^{-1}W_t$ is a Brownian motion on
$S$. When written in components, the resulting dynamics takes the form 
\begin{equation}
 \label{eq:model07}
 \6x_{i,t} = \biggbrak{f(x_{i,t}) +
\frac{\gamma}{2}(x_{i+1,t}-2x_{i,t}+x_{i-1,t}) 
- \frac1N \sum_{j=1}^N f(x_{j,t})}\6t 
+ \sqrt{2\eps} \6\widetilde W_{j,t}
\end{equation} 
where $f(\xi)=-U'(\xi)=\xi-\xi^3$ 
(and the $\widetilde W_{j,t}$ are no longer independent). Note that this is a
discretised version of the mass-conserving Allen--Cahn SPDE
\begin{equation}
 \label{eq:model08}
 \partial_tu(t,x) = \gamma\Delta u(t,x) + f(u(t,x))
 - \frac1L \int_0^L f(u(t,y))\6y
 + \sqrt{2\eps} \, \xi(t,x)
\end{equation} 
with space-time white noise $\xi$ on $S$. 
The nonlocal integral term indeed ensures that the total mass $\int_0^L
u(t,x)\6x$ is conserved. This equation was introduced
in~\cite{Rubinstein_Sternberg_92} in the case without noise, and considered
recently in~\cite{ABBK_15} in the case with noise. 

Systems of the form~\eqref{eq:model02} admit a unique invariant
probability measure with density 
\begin{equation}
 \label{eq:model09}
 \mu(x) = \frac1Z \e^{-V(x)/\eps}\;,
\end{equation} 
where $Z$ is the normalisation constant, and are reversible with respect to
$\mu$. Analogous statements hold true for the system constructed here
(except that $\mu$ is concentrated on the hyperplane $S$). The questions we thus
ask are the following:
\begin{itemiz}
\item 	How long does the system take to relax to equilibrium?
\item 	What are the typical paths taken to achieve equilibrium, when starting
in an atypical configuration?
\item 	Can the system be approximated by a coarse-grained process visiting
only local minima of the potential? What does this coarse-grained process look
like?
\end{itemiz}


\section{Potential landscape}
\label{sec_landscape}


\subsection{The transition graph}
\label{ssec_tgraph}

For a general system of the form~\eqref{eq:model02}, let 
\begin{equation}
 \label{eq:tgraph01}
 \cS = \bigsetsuch{x\in\R^N}{\nabla V_\gamma(x)=0}
\end{equation} 
be the set of all stationary points of $V_\gamma$. A stationary point
$x^\star\in\cS$ is called \emph{non-degenerate} if its Hessian matrix $\nabla^2
V_\gamma(x^\star)$ has a nonzero determinant. We will assume for simplicity
that all stationary points of $V_\gamma$ are nondegenerate (see
however~\cite{BG2010} for results on systems with degenerate stationary
points). 

The \emph{Morse index} of a nondegenerate stationary point $x^\star$ is the
number of negative eigenvalues of the Hessian $\nabla^2V_\gamma(x^\star)$ (i.e.,
the number of directions in which $V_\gamma$ decreases near $x^\star$). For each
$k\in\intint{0}{N}$, let $\cS_k$ denote the set of stationary points of index
$k$. The set $\cS_0$ of local minima of $V_\gamma$ and the set $\cS_1$ of
saddles of index $1$ (or $1$-saddles) are the most important for the stochastic
dynamics for small $\eps$. 

By the stable manifold theorem, each $1$-saddle has a one-dimensional unstable
manifold consisting in two connected components. Along each component, the value
of $V_\gamma$ has to decrease, and therefore (since $V_\gamma$ is confining)
both components have to converge to a local minimum of $V_\gamma$. Let
$\cG=(\cS_0,\cE)$ be the unoriented graph in which two elements of $\cS_0$ are
connected by an edge in $\cE$ if and only if there exists a $1$-saddle
$z\in\cS_1$ whose unstable manifold converges to these local minima. 

Roughly speaking, the stochastic system behaves for small noise intensity
$\eps$ like a Markovian jump process (or continuous-time Markov chain) on
$\cS_0$, with jump rates related to the potential differences between local
minima and $1$-saddles. This is the basic idea implemented
in~\cite[Chapter~6]{FW}, and there are many refinements on which we will
comment in more detail below. 

\begin{figure}[tb]
\begin{center}
\scalebox{0.6}{
\input{fig_unconstained_4}
}
\end{center}
\vspace{-3mm}
\caption[]{Transition graph of the unconstrained system for $N=4$ and
$\gamma=0$. Black and white circles represent respectively coordinates equal to
$1$ and to $-1$. The two configurations $(1,-1,1,-1)$ and $(-1,1,-1,1)$ are not
shown, because they correspond to non-optimal transitions as soon as
$\gamma>0$.}
\label{fig:unconstrained} 
\end{figure}
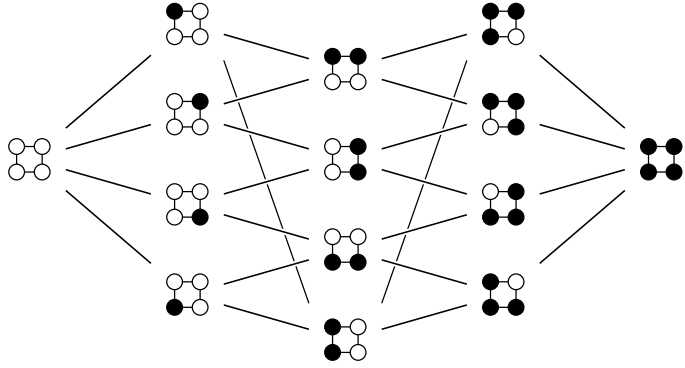

In the case of the potential~\eqref{eq:model01} without constraint, the
potential landscape has been analysed in~\cite{BFG06a}. In particular, the
following properties have been obtained:
\begin{itemiz}
\item 	If $\gamma=0$, the set of stationary points is given by
$\cS=\set{-1,0,1}^N$. The local minima are given by $\cS_0=\set{-1,1}^N$ and
the $1$-saddles are those stationary points that have exactly one coordinate
equal to 0. They connect the local minima obtained by replacing the 0
coordinate by $-1$ or $+1$. Thus the graph $\cG$ is an $N$-dimensional
hypercube, with transitions consisting in the reversal of the sign of one
coordinate, which can be interpreted as spin flips. 

\item 	There exists a critical coupling $\gamma^*(N)$, satisfying
$\gamma^*(N)\geqs\frac14$ for all $N$, such that the transition graph $\cG$
is the same for all $\gamma\in[0,\gamma^*(N))$. Thus the local minima and
allowed transitions are the same for weak positive coupling as in the uncoupled
case. What changes, however, is that some transitions are easier than others
when $\gamma>0$: the systems prefers transitions that minimise the number of
interfaces, that is, the number of nearest neighbours with a different sign
(\figref{fig:unconstrained}). The stochastic dynamics is thus very close to the
one of an Ising model with Glauber spin-flip dynamics. 

\item 	For $\gamma$ increasing beyond $\gamma^*(N)$, the system undergoes a
number of bifurcations that reduce the number of stationary points. In
particular, for $\gamma>1/(2\sin^2(\pi/N))$ the system synchronises: there are
only two local minima given by $\pm(1,1,\dots,1)$, connected by the only
$1$-saddle which is at the origin.
\end{itemiz}

Our aim is now to obtain similar results for the graph $\cG$ of the constrained
system, starting with the uncoupled case $\gamma=0$ and then moving to small
positive $\gamma$. 


\subsection{The uncoupled case}
\label{ssec_gamma0}

We consider in this section the dynamics of the constrained system in the
uncoupled case $\gamma=0$. The above definitions of $\cS_0$, $\cS_1$ and $\cG$
can be adapted to the constained case, either by considering the $N-1$ first
equations in~\eqref{eq:model04}, or by solving a constrained optimisation
problem. In particular, the stationary points have to satisfy 
\begin{equation}
\label{eq:gamma0_01} 
\nabla V_0(x) = \lambda\vone
\end{equation}
for a Lagrange multiplier $\lambda\in\R$ (this is indeed consistent
with~\eqref{eq:model06}). In addition, the constraint $x\in S$ has to be
satisfied. 

In components, the condition~\eqref{eq:gamma0_01} becomes
\begin{equation}
\label{eq:gamma0_02} 
 x_i^3 - x_i = \lambda\;, 
 \qquad
 i = 1,\dots, N\;.
\end{equation} 
Let $\lambdac=\frac{2}{3\sqrt{3}}$. The equation
$\xi^3-\xi=\lambda$ has three real solutions if $\abs{\lambda}<\lambdac$,
two real solutions if $\abs{\lambda}=\lambdac$ and one real solution
otherwise. The last case is incompatible with the constraint $x\in S$, while
the second case can only occur if $N$ is a multiple of $3$, because then the
two solutions of $\xi^3-\xi=\lambda$ have a $(-2:1)$ ratio. 

We henceforth assume that $\abs{\lambda}<\lambdac$, and denote by
$\alpha_0,\alpha_1,\alpha_2$ the distinct roots of $\xi^3-\xi-\lambda$. Then
each $x_i$ solving~\eqref{eq:gamma0_02} has to be equal to one of the 
$\alpha_j$. We let $a_j$ be the number of occurrences of $\alpha_j$, and reorder
the $\alpha_j$ in such a way that $a_0\leqs a_1\leqs a_2$. We denote such a
stationary point by the triple $(a_0,a_1,a_2)$. Observe that we necessarily have
$a_0+a_1+a_2=N$. 

\begin{prop}[Local minima and $1$-saddles for $\gamma=0$]
\label{prop_saddles} 
Assume that $N$ is not a multiple of $3$, and let $x^\star$ be a critical point
with triple $(a_0,a_1,a_2)$. Then 
\begin{itemiz}
\item 	if $2a_1 > a_0 + a_2$, then $x^\star$ is a stationary point of index
$a_0$;
\item 	if $2a_1 < a_0 + a_2$, then $x^\star$ is a stationary point of index
$a_2-1$. 
\end{itemiz}
\end{prop}

We give the proof in Appendix~\ref{ssec_proof_landscape_0}. It is based on the
construction of an orthogonal basis around each stationary point, in which the
Hessian matrix is block-diagonal with blocks of size $3$ at most, so that the
signs of its eigenvalues can be determined.

\begin{remark}
\label{rem_saddles}
The case $2a_1 = a_0 + a_2$ can only occur if $N$ is a multiple of $3$,
because $a_0 + a_2 = N-a_1$ would imply $a_1=N/3$. 
In case $N$ is a multiple of $3$, there exist one-parameter families of
degenerate stationary points~\cite{Dutercq_PhD}. For simplicity we exclude this
situation in all that follows.~$\lozenge$
\end{remark}

Proposition~\ref{prop_saddles} yields the following classification of local
minima and saddles of index~$1$:
\begin{enum}
\item 	Local minima $x^\star\in\cS_0$ necessarily have triple $(0,a,N-a)$
with $N/3 < a \leqs N/2$. 
\item 	Saddles of index $1$ either have triple $(1,a,N-a-1)$ with 
$N/3 < a \leqs (N-1)/2$, or they have triple $(N-2-a,a,2)$ with 
$N/2 - 1 \leqs a \leqs 2$ and $a<N/3$. The latter case can only occur if $N=4$,
and corresponds to the triple $(1,1,2)$.  
\end{enum}

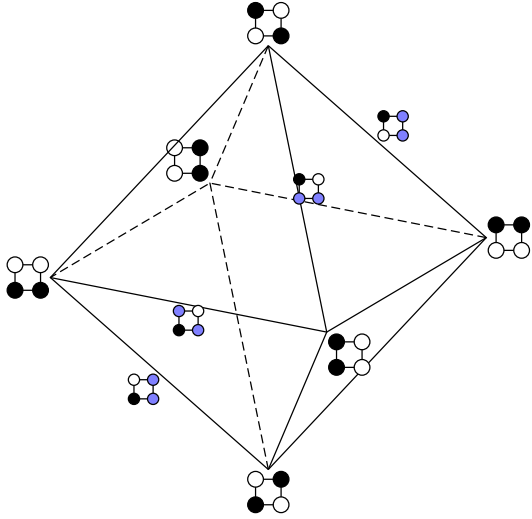
\begin{figure}[tb]
\begin{center}
\scalebox{0.6}{
\input{fig_constrained_4}
}
\end{center}
\vspace{-3mm}
\caption[]{Transition graph of the constrained system for $N=4$ and
$\gamma=0$. Black and white circles represent respectively coordinates equal to
$1$ and to $-1$. A few $1$-saddles associated with edges of the graph are
shown, with blue circles indicating coordinates equal to 0.}
\label{fig:constrained} 
\end{figure}

\begin{example}[The case $N=4$]
\label{ex:N=4}
If $N=4$, then $\cS_0$ contains $6$ points, consisting of all possible
permutations of $(1,1,-1,-1)$. In addition, there are $12$ saddles of index
$1$, consisting of all possible permutations of $(1,-1,0,0)$. Each of these
saddles connects the two local minima obtained by replacing one 0 by $1$ and
the other one by $-1$, and vice
versa~\cite[Section~2.4]{Hun_Masterthesis,Dutercq_PhD}.
The associated transition graph is an octahedron
(\figref{fig:constrained}).~$\blacklozenge$
\end{example}

We will henceforth limit the discussion to the case where $N=2M$ is even,
$N\geqs8$ and $N$ is not a multiple of $3$. Then the $1$-saddles necessarily
correspond to triples of the form $(1,a,N-a-1)$. In order to ease notations, we
write $k_{\max} = \intpart{N/6}$ and  
\begin{itemiz}
\item 	$B_k$ for the set of all local minima with triple $(0,M-k,M+k)$, where
$k\in\intint{0}{k_{\max}}$;
\item 	$C_k$ for the set of all $1$-saddles with triple $(1,M-k,M+k-1)$, where 
$k\in\intint{1}{k_{\max}}$. 
\end{itemiz}
Simple combinatorics shows that the cardinalities of these families are 
\begin{equation}
 \label{eq:gamma0_card}
 \abs{B_0} = \binom{2M}{M}\;, \qquad
 \abs{B_k} = 2\binom{2M}{M+k}\;, \qquad
 \abs{C_k} 
 = \frac{2(2M)!}{1!(M-k)!(M+k-1)!}
\end{equation} 
where $k\in\intint{1}{k_{\max}}$. The factors $2$ are due to the fact that
except for $B_0$, there are always two choices for the signs of coordinates.

One can obtain explicit expressions for the coordinates of all these stationary
points, see~\eqref{eq:alpha_j} in Appendix~\ref{ssec_proof_landscape_0}. Here it
will suffice to know that local minima in $B_0$ simply have $M$ coordinates
equal to $+1$ and $M$ coordinates equal to $-1$. These stationary points are
expected, and admit a simple interpretation in terms of a particle system: we
just associate each coordinate equal to $+1$ with the presence of a particle,
and each coordinates equal to $-1$ with the absence of a particle, that is, a
hole. 

The other families of local minima $B_1,\dots,B_{k_{\max}}$ have more
complicated coordinates, which do not allow for an interpretation as a particle
system. In fact their presence comes a bit as a surprise, so that we will call
them \emph{spurious} configurations. We will however show below that they have a
higher energy than the configurations in $B_0$, and therefore they will not play
an important r\^ole when the system is observed on a sufficiently long
timescale. 

\begin{example}[The case $N=8$]
If $N=8$, there are two families of local minima $B_0$ and $B_1$, and one
family of $1$-saddles $C_1$ (\figref{fig:transition_8}). 
\begin{itemiz}
\item 	The family of local minima $B_0$ corresponds to the triple $(0,4,4)$,
and contains all points that have $4$ coordinates equal to $+1$ and $4$
coordinates equal to $-1$. They can thus be interpreted as configurations with
$4$ particles and $4$ holes. 

\item 	The family of local minima $B_1$ corresponds to the triple $(0,3,5)$. It
contains all points with $3$ coordinates equal to $5/\sqrt{19}$ and $5$
coordinates equal to $-3/\sqrt{19}$, as well as all configurations with
opposite signs. 

\item 	The family of $1$-saddles $C_1$ corresponds to the triple $(1,3,4)$.  It
contains all points with $1$ coordinate equal to $-1/\sqrt{7}$, $3$ coordinates
equal to $3/\sqrt{7}$ and $4$ coordinates equal to $-2/\sqrt{7}$, as well as all
configurations with opposite signs.~$\blacklozenge$ 
\end{itemiz}
\end{example}

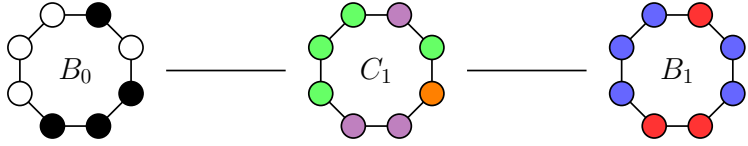
\begin{figure}[tb]
\begin{center}
\scalebox{0.8}{
\input{fig_transition_8}
}
\end{center}
\vspace{-3mm}
\caption[]{Example of transition rules for $N=8$. The coordinates for
family $B_0$ are $\tikzcircle{black}{2pt}=1$ and 
$\tikzcircle{white}{2pt}=-1$. Those for $C_1$ are 
$\tikzcircle{red!50!yellow}{2pt}=-1/\sqrt{7}$, 
$\tikzcircle{violet!50}{2pt}=3/\sqrt{7}$ and 
$\tikzcircle{green!60}{2pt}=-2/\sqrt{7}$. Those for $B_1$ are 
$\tikzcircle{red!80}{2pt}=5/\sqrt{19}$ and 
$\tikzcircle{blue!60}{2pt}=-3/\sqrt{19}$.}
\label{fig:transition_8} 
\end{figure}

Now that we have determined all stationary points in $\cS_0$ and $\cS_1$, we
have to find the structure of the transition graph $\cG=(\cS_0,\cE)$. In other
words, we have to determine which local minima are connected by a given
$1$-saddle. This question is answered in the following result. 

\begin{theorem}[Transition graph for $\gamma=0$]
\label{thm:transition_graph}
Each $1$-saddle in $C_k$ connects exactly one local minimum in $B_{k-1}$ with
one local minimum in $B_k$. More precisely, if the coordinates of the
saddle have values $\alpha'_0, \alpha'_1, \alpha'_2$, and those of the local
minima are respectively $\alpha_1,\alpha_2$ and $\alpha''_1, \alpha''_2$, 
then the connection rules
are given by 
\begin{align}
\nonumber
&\alpha_1 \longleftrightarrow \alpha'_0 \longleftrightarrow \alpha''_2 
&& \text{$1$ coordinate\;,} \\
\label{eq:gamma0_rules} 
&\alpha_1 \longleftrightarrow \alpha'_1 \longleftrightarrow \alpha''_1 
&& \text{$M-k$ coordinates\;,} \\
&\alpha_2 \longleftrightarrow \alpha'_2 \longleftrightarrow \alpha''_2 
&& \text{$M+k-1$ coordinates\;.}
\nonumber
\end{align} 
\end{theorem}

We give the proof in Appendix~\ref{ssec_proof_landscape_0}. It is based on the
construction of two continuous paths connecting a given point in $C_k$ to one
point in $B_{k-1}$ and one point in $B_k$, such that the potential decreases
along the path when moving away from the saddle. \figref{fig:transition_8}
illustrates the connection rule in the case $N=8$. See also~\cite[Fig.~5]{BD15}.

Using the relations~\eqref{eq:gamma0_card}, one easily checks that the number
of saddles in $C_k$ is indeed equal to the number of allowed connections
between elements in $B_{k-1}$ and $C_k$ as well as $B_k$ and $C_k$.


\subsection{The case of weak positive coupling}
\label{ssec_gamma_small}

It follows from basic perturbation arguments that the transition graph $\cG$
will persist for small positive coupling intensity $\gamma$. Indeed, if we
assume that $N$ is not a multiple of $3$, then all stationary points for
$\gamma=0$ are nondegenerate, so that the implicit function theorem shows that
they still exist for small positive coupling, and move at most by a distance of
order $\gamma$. In addition, perturbation results for the eigenvalues of
matrices such as the Bauer--Fike theorem (see for
instance~\cite{Golub_VanLoan_13}) show that the signature of nondegenerate
stationary points does not change for small $\gamma$. Finally, the proof of
Theorem~\ref{thm:transition_graph} essentially relies on the
relation~\eqref{eq:lp0_08}, whose coefficients depend continuously on $\gamma$.

The drawback of this argument is that while it shows that for any $N<\infty$,
there exists a critical coupling $\gamma^*(N)>0$ such that the transition graph
does not change for $0\leqs\gamma<\gamma^*(N)$, it does not yield a good control
on the critical coupling as $N\to\infty$. To obtain a lower bound on
$\gamma^*(N)$ which is uniform in $N$ (at least for $k$ fixed), we adapt
from~\cite{BFG06a} an argument based on symbolic dynamics to obtain the
following result. 

\begin{theorem}[Persistence of the transition graph for small positive $\gamma$]
 \label{thm:persistence} 
There exists a constant $c>0$, independent of $N$, such that the stationary
points of the families $B_k$ and $C_k$ persist for 
\begin{equation}
 \label{eq:wpc01}
 \gamma \leqs c \biggpar{\frac16 - \frac{k}{N}}^2\;,
\end{equation} 
without changing their index. In the particular case of stationary points of the
family $B_0$, we have the sharper result that they persist at least as long
as  $\gamma < \frac73 - \sqrt{5} \simeq 0.097$. 
\end{theorem}

The proof is given in Appendix~\ref{ssec_proof_landscape_gamma}. It 
also provides a criterion allowing to sharpen the bound~\eqref{eq:wpc01} for
families other than $B_0$, cf.~\eqref{eq:plg_42}, which however is not essential
in what follows.  

The important aspect of this result is that all families of stationary points
$B_k$ or $C_k$ with $\frac kN$ bounded away from $\frac16$ are ensured to exist
up to a positive critical coupling independent of $N$. Only stationary points
with $k = \frac N6 - \order{N}$  might disappear at a critical $\gamma$ which
vanishes in the large-$N$ limit. 


\section{Metastable hierarchy}
\label{sec_meta}


Now that the structure of the transition graph $\cG$ is understood, we have
access to information on timescales of the metastable process. A convenient way
of doing this relies on the concept of~\emph{metastable hierarchy}, which is an
ordering of the local minima from deepest to shallowest. We summarise this
construction in Section~\ref{ssec_def_hierarchy}, before applying it to our case
in Section~\ref{ssec_hierarchy_Bk}. A more refined hierarchy can be obtained for
small positive coupling $\gamma$ among the local minima of the family $B_0$,
which have a particle interpretation; we do this in
Section~\ref{ssec_hierarchy_B0}. 


\subsection{Metastable hierarchy and Eyring--Kramers law}
\label{ssec_def_hierarchy}

We consider in this section a general reversible diffusion process in $\R^N$ of
the form~\eqref{eq:model02}, with potential $V$ of class $\cC^2$. 

\begin{definition}[Communication height]
Let $x^\star$ be a local minimum of $V$ and let $A\subset\R^N$. The
\emph{communication height}\/ from $x^\star$ to $A$ is the nonnegative number 
\begin{equation}
 \label{eq:def_com_height}
 H(x^\star, A) = \inf_{\gamma:x^\star\to A} \sup_{t\in[0,1]} V(\gamma(t)) -
V(x^\star)\;,
\end{equation} 
where the infimum runs over all continuous paths $\gamma:[0,1]\to\R^N$ such that
$\gamma(0)=x^\star$ and $\gamma(1)\in A$. Any path $\gamma$
realising~\eqref{eq:def_com_height} is called a \emph{minimal path}\/ from
$x^\star$ to $A$. 
\end{definition}

The communication height measures how high one cannot avoid climbing in the
potential landscape to go from $x^\star$ to $A$. Assuming $A$ does not
intersect the basin of attraction of $x^\star$ and all stationary points of $V$
are nondegenerate, it is not difficult to show that the supremum
in~\eqref{eq:def_com_height} is reached at a $1$-saddle $z^\star$ of $V$ (see
for instance~\cite[Section~2]{BG2010}). In that case, one has
$H(x^\star,A) = V(z^\star) - V(x^\star)$. 

A notion of metastable order of local minima was introduced in~\cite{BGK}. In
our case, due to the fact that many minima have the same or almost the same
potential value, we introduce the following generalisation of this concept to
partitions of the set of local minima. Typically, we will apply this definition
to cases where the points in each element of the partition have approximately or
exactly the same potential height. 

\begin{definition}[Metastable hierarchy of a partition]
 \label{def:metastable_hierarchy}
A partition $\cS_0 = P_1 \cupdot P_2 \cupdot \dots \cupdot P_m$ of the set
$\cS_0$ of local minima of $V$ is said to form a \emph{metastable hierarchy} if
there exists a constant $\theta>0$ such that for all $k\in\intint{2}{m}$, one
has 
\begin{equation}
 \label{eq:def_meta_hierarchy}
 H \biggpar{x^\star, \bigcup_{i=1}^{k-1}P_i} 
 \leqs \min_{y^\star\in P_\ell} 
 H \biggpar{y^\star, \bigcup_{i=1}^kP_i \setminus P_\ell} -
\theta 
\end{equation} 
for all $x^\star\in P_k$ and all $\ell\in\intint{1}{k-1}$. In this case, we
write 
\begin{equation}
 \label{eq:not_meta_hierarchy}
 P_1 \prec P_2 \prec \dots \prec P_m\;.
\end{equation} 
\end{definition}

In words, it is easier, starting in any point in $P_k$, to reach a lower-lying
set $P_\ell$ in the hierarchy than it is, starting in such an $P_\ell$, to reach
any other set among $P_1, \dots P_k$. A graphical way of constructing the
hierarchy relies on the so-called \emph{disconnectivity tree}
\cite{Cameron_Vanden-Eijnden_2014}; it is illustrated
in~\figref{fig:disconnectivity} in a simple case where all $P_k=\set{x^\star_k}$
are singletons. The leaves of the tree have coordinates
$(x^\star_k,V(x^\star_k))$; each leaf is connected to the lowest saddle
reachable from it, and the procedure is repeated after discarding the shallower
local minimum whenever two branches join.  

\begin{figure}[tb]
\begin{center}
\scalebox{0.8}{
\input{fig_disconnectivity}
}
\end{center}
\vspace{-5mm}
\caption[]{Example of a $4$-well potential, with its disconnectivity tree. The
metastable order is given by $x^\star_1 \prec x^\star_2 \prec x^\star_3 \prec
x^\star_4$. The communication heights $H_k =
H(x^\star_k,\set{x^\star_1,\dots,x^\star_{k-1}}) = V(z^\star_k)-V(x^\star_k)$
provide the Arrhenius exponents for mean transition times and small eigenvalues
of the generator $\cL$. Prefactors in the Eyring--Kramers law~\eqref{eq:BGK01}
are given in terms of second derivatives of the potential at the local minima
$x^\star_k$ and $1$-saddles $z^\star_k$.}
\label{fig:disconnectivity} 
\end{figure}
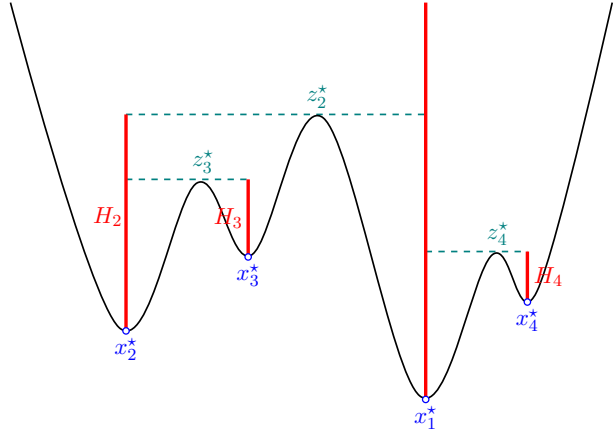

In the particular case where all $P_k$ are singletons, the following
result by Bovier, Gayrard and Klein connects the metastable hierarchy with
certain first-hitting times and with small eigenvalues of the infinitesimal
generator $\cL=\eps\Delta - \nabla V(x)\cdot\nabla$ of the diffusion. 

\begin{theorem}[Eyring--Kramers law for nondegenerate potentials \cite{BGK}]
\label{thm:BGK} 
Assume the local minima of $V$ admit a metastable order $x^\star_1 \prec \dots
\prec x^\star_m$. For each $k\in\intint{1}{m}$, denote by $\tau_k$ the
first-hitting time of the $\eps$-neighbourhood of $\set{x^\star_1, \dots,
x^\star_k}$, and let $\lambda_k$ by the $k$th smallest eigenvalue of $-\cL$.
Assume further that for each $k$, there is a unique $1$-saddle $z^\star_k$ such
that any minimal path from $x^\star_k$ to $\set{x^\star_1, \dots,
x^\star_{k-1}}$ reaches communication height only at $z^\star_k$. Then for each
$k\in\intint{2}{m}$, one has 
\begin{equation}
 \label{eq:BGK01}
 \expecin{x^\star_k}{\tau_{k-1}} 
 = \frac{2\pi}{\abs{\lambda_-(z^\star_k)}} 
 \sqrt{\frac{\abs{\det\nabla^2V(z^\star_k)}}{\det\nabla^2V(x^\star_k)}}
\e^{\brak{V(z^\star_k)-V(x^\star_k)}/\eps} 
\bigbrak{1 + \Order{\eps^{1/2}\abs{\log\eps}^{3/2}}}\;, 
\end{equation} 
where $\nabla^2 V(x)$ denotes the Hessian matrix of $V$ at $x$, and
$\lambda_-(z^\star_k)$ is the unique negative eigenvalue of
$\nabla^2V(z^\star_k)$. Furthermore, $\lambda_1=0$ and there exists a constant
$\theta_1>0$ such that 
\begin{equation}
 \label{eq:BGK02}
 \lambda_k = \frac{1}{\expecin{x^\star_k}{\tau_{k-1}}}
\bigbrak{1+\Order{\e^{-\theta_1/\eps}}}
\end{equation} 
holds for all $k\in\intint{2}{m}$. 
\end{theorem}

This result tells us in particular that if the system starts at a stationary
point at the end of the metastable hierarchy, it will spend longer and longer
amounts of time going down the hierarchy (possibly visiting other local minima
in between), before reaching the ground state $x^\star_1$. In particular, the
spectral gap $\lambda_2-\lambda_1=\lambda_2$ of the system, which gives the
exponential rate of convergence to equilibrium, depends to leading order only on
the second local minimum in the hierarchy $x^\star_2$, and on the saddle
$z^\star_2$ connecting it to the ground state.


\subsection{Hierarchy on the families $B_k$}
\label{ssec_hierarchy_Bk}

Unfortunately, Theorem~\ref{thm:BGK} does not apply to our situation, because
one cannot find a hierarchy for singletons. This is due to the fact that the
potential $V_\gamma$ has many symmetries, and therefore many stationary points
have the same potential height, preventing us from
fulfilling~\eqref{eq:def_meta_hierarchy} with a positive $\theta$. In
particular, in the uncoupled case $\gamma=0$, the system is invariant under
the group $G=\mathfrak{S}_N\times\Z_2$, where $\mathfrak{S}_N$ is the symmetric
group describing permutations of the $N$ coordinates, and the factor
$\Z_2=\Z/2\Z$ accounts for the $x\mapsto-x$ symmetry. The families $B_k$ and
$C_k$ each form a group orbit under $G$, that is, they are equivalence classes
of the form $\setsuch{gx}{g\in G}$.

However, we will be able to draw on results of~\cite{BD15}, which generalise
Theorem~\ref{thm:BGK} to Markovian jump processes invariant under a group of
symmetries, and the extension of these results to diffusion
processes~\cite{Dutercq_PhD,Dutercq_in_preparation}. In particular,
\cite[Thm~3.2]{BD15} shows that if the system starts with an initial
distribution which is uniform on some $B_k$, then a very similar result to
Theorem~\ref{thm:BGK} holds true. The only difference is that the prefactor in
the Eyring--Kramers law~\eqref{eq:BGK01} has to be multiplied by a factor which
can be explicitly computed in terms of stabilisers of the group orbits. 

\begin{figure}[tb]
\begin{center}
\scalebox{0.8}{
\input{fig_BC2}
}
\end{center}
\vspace{-5mm}
\caption[]{Value of the potential $V_0$ along a path $B_0 \to C_1 \to B_1 \to
\dots$ in the case $N=20$ (not to scale). The associated disconnectivity tree
shows that the $B_k$ are indeed in metastable order. Thus the long-time dynamics
will concentrate on the set $B_0$ of particle--hole configurations.}
\label{fig:BC} 
\end{figure}
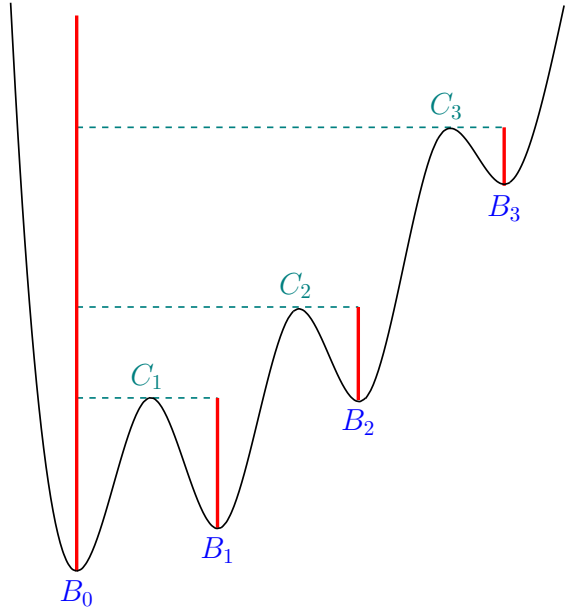

The following result provides a metastable order on the $B_k$, which is exactly
what is required to apply the theory
from~\cite{BD15,Dutercq_PhD,Dutercq_in_preparation} in the uncoupled case. 

\begin{theorem}[Metastable hierarchy on the $B_k$]
\label{thm:hierarchy_Bk} 
If $\gamma=0$, then the families $B_k$ satisfy a metastable order given by 
\begin{equation}
\label{eq:Bk01} 
 B_0 \prec B_1 \prec \dots \prec B_{k_{\max}}\;.
\end{equation} 
Furthermore, any minimal path from $B_k$ to $B_{k-1}$ reaches communication
height only on saddles in $C_k$. The hierarchy~\eqref{eq:Bk01} still applies for
sufficiently small positive $\gamma$, the only difference being that all points
inside a given $B_k$ do not necessarily have the same potential value.  
\end{theorem}

We give the proof in Appendix~\ref{ssec_proof_Bk}. The situation is illustrated
in~\figref{fig:BC}. As $k$ increases from $0$ to $k_{\max}$, the potential
height of the $B_k$ increases, while the barrier height between $C_k$ and
$B_k$ decreases. See Appendix~\ref{ssec_proof_Bk} for explicit expressions for
these potential values. Applying~\cite[Thm~3.2]{BD15}, we obtain in particular
the following result.

\begin{cor}
\label{cor:Bk}
For $k\in\intint{1}{k_{\max}}$, let $\tau_{k-1}$ be the first-hitting time
of the $\eps$-neighbourhood of $B_0\cup\dots\cup B_{k-1}$. If the initial
distribution $\mu$ of the system is concentrated on $B_k\cup\dots\cup
B_{k_{\max}}$ and invariant under $G$, then for $\gamma=0$ one has 
\begin{equation}
 \label{eq:Bk02}
 \expecin{\mu}{\tau_{k-1}} 
 = \frac{2\pi}{\abs{\lambda_-(z^\star_k)}(M+k)} 
 \sqrt{\frac{\abs{\det\nabla^2V_0(z^\star_k)}}{\det\nabla^2V_0(x^\star_k)}}
\e^{\brak{V_0(z^\star_k)-V_0(x^\star_k)}/\eps} 
\bigbrak{1 + \Order{\eps^{1/2}\abs{\log\eps}^{3/2}}}\;, 
\end{equation} 
where $x^\star_k$ is any local minimum in $B_k$, $z^\star_k$ is any saddle in
$C_k$, and $M=N/2$.  
\end{cor}
\begin{proof}
Theorem~3.2 in~\cite{BD15} shows that in the case of a symmetric initial
distribution, the usual Eyring--Kramers formula~\eqref{eq:BGK01} has to be
multiplied by the factor $\abs{G_{x^\star_k} \cap
G_{x^\star_{k-1}}}/\abs{G_{x^\star_k}}$, where $G_x=\setsuch{g\in G}{g(x)=x}$
is the stabiliser of $x$. If $k\geqs1$, then $\abs{G_{x^\star_k}}$ is the number
of permutations that leave invariant any element in $B_k$, and is equal to
$(M-k)!(M+k)!$. Similarly, $\abs{G_{x^\star_k} \cap G_{x^\star_{k-1}}}$ is the
number of permutations leaving invariant any two elements in $B_k$ and
$B_{k-1}$ connected in the transition graph $\cG$, which is equal to
$(M-k)!(M+k-1)!$.  
\end{proof}

Note the extra factor $(M+k)^{-1}$ in~\eqref{eq:Bk02}. In fact, $M+k$ is also
the number of saddles in $C_k$ that are connected with any given element of
$B_k$ (cf.~\cite[Eq.~(2.25)]{BD15}). The interpretation of this factor is that
since the system has $M+k$ different ways to make a transition from a given
$x^\star_k\in B_k$ to $B_{k-1}$, the transition time is divided by this factor. 

The above result will still apply for small positive coupling, but with a more
complicated expression for the prefactor. This is because the system is no
longer invariant under $\mathfrak{S}_N\times\Z_2$, but under the smaller group
$\mathfrak{D}_N\times\Z_2$, where $\mathfrak{D}_N$ is the dihedral group of
symmetries of a regular $N$-gon. The important aspect for us is that we still
have a control of the time needed to reach the family of stationary points
$B_0$, which lie at the bottom of the hierarchy and have an interpretation in
terms of particle--hole configurations. The dynamics among configurations in
$B_0$ is much slower than the relaxation towards $B_0$, because it involves
crossing the potential barrier from $B_0$ to $B_1$ via $C_1$. We will analyse it
in more detail in the next section.


\subsection{Hierarchy on $B_0$ and particle interpretation}
\label{ssec_hierarchy_B0}

We assume in this section that $0<\gamma\ll\gammac$, where $\gammac$ is the
critical coupling below which all stationary points in $B_0$, $B_1$ and $C_1$
exist without bifurcating. The central observation in order to classify points
in $B_0$ is that if $x^\star(\gamma)$ is any critical point of $V_\gamma$, then 
\begin{equation}
 \label{eq:B0_01}
 V_\gamma\bigpar{x^\star(\gamma)} = V_0\bigpar{x^\star(0)} 
 + \frac\gamma4 \sum_{i=1}^N\bigpar{x^\star_{i+1}(0) - x^\star_i(0)}^2 +
\Order{\gamma^2}\;.
\end{equation} 
This is because $V_0(x^\star(\gamma)) = V_0(x^\star(0)) + \Order{\gamma^2}$,
as the first-order term in $\gamma$ vanishes 
since $\nabla V_0(x^\star(0))=\lambda\vone$ is orthogonal to $x^\star(\gamma) -
x^\star(0)$, which belongs to the hyperplane $S$. The first term on the
right-hand side of~\eqref{eq:B0_01} is constant on each $B_k$ and each $C_k$.
The second term is determined by the number of nearest-neighbour coordinates of
$x^\star(0)$ that are different, which we are going to call \defwd{interfaces}
of the configuration. 

In particular, if $x^\star(0) \in B_0$, we know that all its components have
values $\pm1$. We define its number of interfaces as 
\begin{equation}
 \label{eq:B0_01a}
 I_{1/-1}(x^\star) = 
 \sum_{i=1}^N 1_{\set{x^\star_i(0) \neq x^\star_{i+1}(0)}}
\end{equation} 
so that we have 
\begin{equation}
 \label{eq:B0_02}
 V_\gamma\bigpar{x^\star(\gamma)} = V_0\bigpar{x^\star(0)} 
 + \gamma I_{1/-1}(x^\star) + \Order{\gamma^2}
\end{equation} 
where $V_0(x^\star(0)) = -\frac14 N$. Furthermore, we define the number of
interfaces at site $i$ as 
\begin{equation}
 \label{eq:B0_02a}
 I_{1/-1}(x^\star,i) = 1_{\set{x^\star_{i-1}(0) \neq x^\star_{i}(0)}} +
1_{\set{x^\star_i(0) \neq x^\star_{i+1}(0)}} \in\set{0,1,2}\;.
\end{equation}
Interpreting each $1$ as a particle and each $-1$ as a hole, it is natural to
introduce the following terminology:
\begin{itemiz}
\item 	a site $i$ with $2$ interfaces will be called an \emph{isolated}
particle or hole;
\item 	a sequence of at least $2$ contiguous particles or holes will be called
a \emph{cluster};
\item 	a site with $1$ interface lies at the \emph{boundary of a cluster};
\item 	a site without interface belongs to the \emph{bulk of a cluster}.
\end{itemiz}

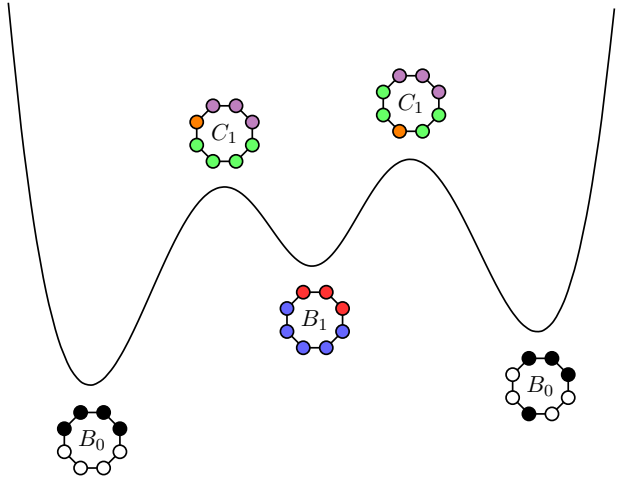
\begin{figure}[tb]
\begin{center}
\scalebox{0.8}{
\input{fig_transition_BC}
}
\end{center}
\vspace{-5mm}
\caption[]{Example of an allowed transition, from a configuration in $B_0$ with
two interfaces to a configuration in $B_0$ with $4$ interfaces. The net effect
is that a particle has hopped by two sites.}
\label{fig:transition_BC} 
\end{figure}

\begin{lemma}
\label{lem:B0} 
Let $x^\star$ be a critical point in $B_0$ and write $M=\frac N2 \geqs 4$. 
Then the following properties hold.
\begin{enum}
\item 	The total number of interfaces $I_{1/-1}(x^\star)$ is even.
\item 	If $I_{1/-1}(x^\star)=2$, then $x^\star$ consists in a cluster of $M$
particles and a cluster of $M$ holes.
\item 	If $I_{1/-1}(x^\star)>M$, then $x^\star$ has at least one isolated
site. 
\item 	Among the $x^\star\in B_0$ with $I_{1/-1}(x^\star)\in\intint{4}{M}$, 
there exist both configurations with isolated sites and configurations without
isolated sites.
\end{enum}
\end{lemma}
\begin{proof}
Denote by $\Nc$ the number of clusters, by $\Ni$ the number of isolated
sites, and by $p=I_{1/-1}(x^\star)$ the number of interfaces. Then we have
$p=\Nc+\Ni$, which is necessarily even. Since clusters have at least two sites, 
$N\geqs 2\Nc+\Ni$, implying $\Nc\leqs N-p$ and thus $\Ni\geqs 2p -N$. Thus if
$p>M$, then $\Ni>0$. If $4\leqs p\leqs M$, then a possible configuration
consists in $p-2$ clusters of size $2$, leaving at least $4$ sites that can be
split into $2$ more clusters. Another possibility is to have $p-2$ isolated
sites, leaving at least $N-2$ sites that can again be split into $2$ clusters.
If $p=2$, we necessarily have $2$ clusters of equal size. 
\end{proof}

This result motivates the following notation for configurations in $B_0$: 
\begin{itemiz}
\item 	$A_2$ denotes the set of all configurations with interface number
$I_{1/-1}(x^\star) = 2$; 
\item 	for even $p\in\intint{4}{M}$, $A_p$ denotes the set of all
configurations $x^\star\in B_0$ with $p$ interfaces having at least one isolated
site, and $A'_p$ denotes the set of configurations with $p$ interfaces having no
isolated site;
\item 	for even $p\in\intint{M+1}{N}$, $A_p$ denotes the set of configurations
with $p$ interfaces (which all have at least one isolated site). 
\end{itemiz}

\begin{table}[tb]
\begin{center}
\begin{tabular}{|c||c|c|c|c|}
\hline
\vrule height 14pt depth 6pt width 0pt
 & Transition & $\Delta p$ & $H^{(1)}\bigpar{x^\star_1(0), x^\star_2(0)}$ &
Saddle
\\
\hline
\hline
\hlinespace
 \rom{1} &
 \raisebox{-1mm}{\tikz{
 \bdots{-0.6} \bpart{0} \bpart{0.3} \bpart{0.6}
 \bdots{1.2} \bhole{1.8} \bhole{2.1} \bhole{2.4} \bdots{3}
 \barrow{0.3}{2.1}
 }}
 & $+4$ 
 & $\dfrac{10M^2 - 36M + 36 -3p}{4(M^2-3M+3)}$ 
 & $[0,2,p+2]$ \\
\hline
\topspace
 \rom{2}.a &
 \raisebox{-1mm}{\tikz{
 \bdots{-0.6} \bpart{0} \bpart{0.3} \bpart{0.6}
 \bdots{1.2} \bhole{1.8} \bhole{2.1} \bpart{2.4} \bdots{3}
 \barrow{0.3}{2.1}
 }}
 & $+2$ 
 & $\dfrac{2(M-3)^2 -3p}{4(M^2-3M+3)}$ 
 & $[0,2,p]$ \vspace{-2mm} \\ 
 \rom{2}.b &
 \raisebox{-1mm}{\tikz{
 \bdots{-0.6} \bpart{0} \bpart{0.3} \bhole{0.6}
 \bdots{1.2} \bhole{1.8} \bhole{2.1} \bhole{2.4} \bdots{3}
 \barrow{0.3}{2.1}
 }} 
 & 
 &
 & \\ 
\bottomspace
 \rom{2}.c &
 \raisebox{-1mm}{\tikz{
 \bdots{-0.6} \bpart{0} \bpart{0.3} 
 \bhole{0.6} \bhole{0.9} \bdots{1.5}
 \narrow{0.3}{0.6}
 }} 
 & 
 &
 & \\
\hline
\hlinespace
 \rom{3} &
 \raisebox{-1mm}{\tikz{
 \bdots{-0.6} \bpart{0} \bpart{0.3} \bhole{0.6}
 \bdots{1.2} \bhole{1.8} \bhole{2.1} \bpart{2.4} \bdots{3}
 \barrow{0.3}{2.1}
 }} 
 & $0$ 
 & $\dfrac{-2M^2 +6M -3p}{4(M^2-3M+3)}$ 
 & $[1,1,p-1]$ \\
\hline
\topspace
 \rom{4}.a &
 \raisebox{-1mm}{\tikz{
 \bdots{-0.6} \bpart{0} \bpart{0.3} \bpart{0.6}
 \bdots{1.2} \bpart{1.8} \bhole{2.1} \bpart{2.4} \bdots{3}
 \barrow{0.3}{2.1}
 }}
 & $0$ 
 & $\dfrac{-6M^2 +12M -3p}{4(M^2-3M+3)}$ 
 & $[0,2,p-2]$ \vspace{-2mm} \\ 
 \rom{4}.b &
 \raisebox{-1mm}{\tikz{
 \bdots{-0.6} \bhole{0} \bpart{0.3} \bhole{0.6}
 \bdots{1.2} \bhole{1.8} \bhole{2.1} \bhole{2.4} \bdots{3}
 \barrow{0.3}{2.1}
 }} 
 & 
 &
 & \\ 
 \rom{4}.c &
 \raisebox{-1mm}{\tikz{
 \bdots{-0.6} \bpart{0} \bpart{0.3} 
 \bhole{0.6} \bpart{0.9} \bdots{1.5}
 \narrow{0.3}{0.6}
 }} 
 & 
 &
 & \\
 \bottomspace
 \rom{4}.d &
 \raisebox{-1mm}{\tikz{
 \bdots{-0.6} \bhole{0} \bpart{0.3} 
 \bhole{0.6} \bhole{0.9} \bdots{1.5}
 \narrow{0.3}{0.6}
 }} 
 & 
 &
 & \\
 \cline{1-3}
 \topspace
 \rom{5}.a &
 \raisebox{-1mm}{\tikz{
 \bdots{-0.6} \bpart{0} \bpart{0.3} \bhole{0.6}
 \bdots{1.2} \bpart{1.8} \bhole{2.1} \bpart{2.4} \bdots{3}
 \barrow{0.3}{2.1}
 }} 
 & $-2$ 
 &
 & \\ 
 \rom{5}.b &
 \raisebox{-1mm}{\tikz{
 \bdots{-0.6} \bhole{0} \bpart{0.3} \bhole{0.6}
 \bdots{1.2} \bhole{1.8} \bhole{2.1} \bpart{2.4} \bdots{3}
 \barrow{0.3}{2.1}
 }} 
 & 
 &
 & \\ 
 \bottomspace
 \rom{5}.c &
 \raisebox{-1mm}{\tikz{
 \bdots{-0.6} \bhole{0} \bpart{0.3} 
 \bhole{0.6} \bpart{0.9} \bdots{1.5}
 \narrow{0.3}{0.6}
 }} 
 & 
 &
 & \\
 \cline{1-3}
 \topspace
 \bottomspace
 \rom{6} &
 \raisebox{-1mm}{\tikz{
 \bdots{-0.6} \bhole{0} \bpart{0.3} \bhole{0.6}
 \bdots{1.2} \bpart{1.8} \bhole{2.1} \bpart{2.4} \bdots{3}
 \barrow{0.3}{2.1}
 }} 
 & $-4$ 
 &
 & \\ 
\hline
\end{tabular}
\end{center}
\caption[]{List of allowed transitions between elements of $B_0$, viewed as
a particle moving into a hole. The different columns show, respectively,
the type of transition, the change $\Delta p$ of the number of interfaces, the
first-order correction to the communication height, and the numbers of
interfaces of types $\alpha'_0/\alpha'_1$, $\alpha'_0/\alpha'_2$ and
$\alpha'_1/\alpha'_2$ of the highest saddle encountered during the
transition (cf.~Appendix~\ref{ssec_proof_B0}).}
\label{table:transitions_B0} 
\end{table}

We now need to determine the communication heights between configurations in
these different sets for small positive $\gamma$. For this, we have to take into
account the fact that any transition between two configurations in $B_0$
involves crossing two $1$-saddles in $C_1$, separated by an element of $B_1$
(\figref{fig:transition_BC}). The communication height will thus be determined
by the highest of the two saddles. Examining the different possible cases
yields the following result, which is proved in Appendix~\ref{ssec_proof_B0}. 

\begin{prop}[Transitions between configurations in $B_0$]
\label{prop:transitions_B0} 
Let $x^\star_1(\gamma), x^\star_2(\gamma) \in B_0$ be two particle/hole
configurations, and denote by $p=I_{1/-1}(x^\star_1(0))$ the number of
interfaces of $x^\star_1(0)$. Then a transition between these configurations is
possible if and only if $x^\star_2(0)$ is obtained by interchanging a particle
and a hole in $x^\star_1(0)$. The interface number of $x^\star_2(0)$ satisfies 
\begin{equation}
 \label{eq:B0_03} 
 I_{1/-1}\bigpar{x^\star_2(0)} \in \set{p-4,p-2,p,p+2,p+4}\;.
\end{equation} 
The communication height from $x^\star_1(\gamma)$ to $x^\star_2(\gamma)$ admits
the expansion  
\begin{equation}
 \label{eq:B0_04}
 H\bigpar{x^\star_1(\gamma), x^\star_2(\gamma)}
 = H^{(0)} + \gamma H^{(1)}\bigpar{x^\star_1(0), x^\star_2(0)} +
\Order{\gamma^2}\;,
\end{equation} 
where 
\begin{equation}
 \label{eq:B0_05} 
 H^{(0)} = V_0(C_1) - V_0(B_0) 
 = \frac{M(M-1)}{4(M^2-3M+3)}
\end{equation} 
depends only on $M=\frac N2$, while $H^{(1)}(x^\star_1(0), x^\star_2(0))$ also
depends on $p$ and on the number of interfaces of the two exchanged sites as
detailed in Table~\ref{table:transitions_B0}.
\end{prop}

Table~\ref{table:transitions_B0} shows that all allowed transitions between
particle/hole configurations have simple physical interpretations. In
particular, only the last four types of transitions decrease the number of
interfaces. Types \rom{5}.b and \rom{5}.c can be viewed as an isolated particle
merging with another particle (isolated or at the boundary of a cluster), type
\rom{5}.a as a particle splitting from another one to fill a hole between two
particles, and type \rom{6} as an isolated particle jumping into a hole
between two particles. Types \rom{1} and \rom{2} are just the reversed versions
of types \rom{6} and \rom{5}, while all transitions of type \rom{3} and \rom{4}
are their own reverse. 

\begin{figure}[tb]
\begin{center}
\scalebox{0.7}{
\input{fig_B0_N8_2}
}
\end{center}
\vspace{-3mm}
\caption[]{Minimal transitions between particle/hole configurations in $B_0$ for
$N=8$. Arrows indicate transitions that decrease the energy, and are labelled
according to Table~\ref{table:transitions_B0}. Each node displays only one
representative of an orbit for the group action of $\mathfrak{D}_N\times\Z_2$.
The other elements of an orbit are obtained by applying rotations, reflections
and interchanging particles and holes. Blue nodes represent stationary points in
$B_1$. Not shown are transitions within the families $A_p$ and $A'_p$, which are
of type \rom{3} or \rom{4}. 
}
\label{fig:B0_N8} 
\end{figure}
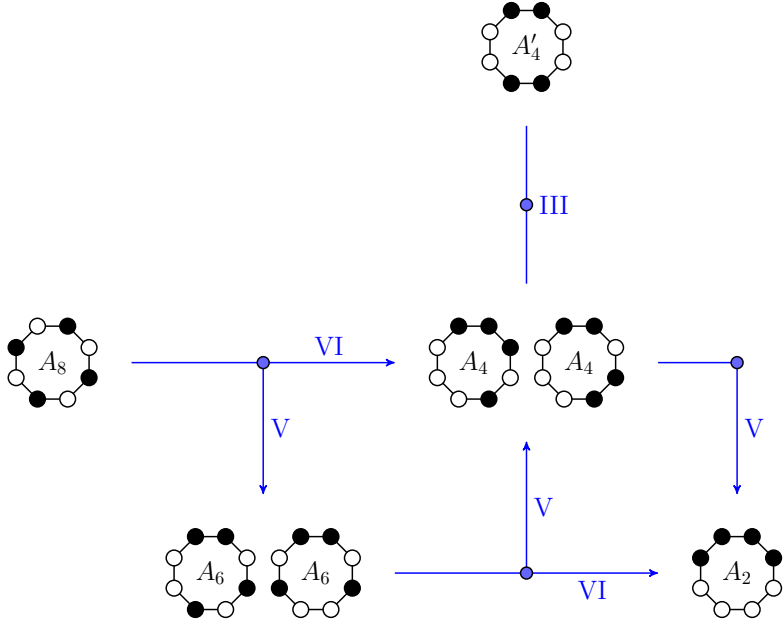

\figref{fig:B0_N8} shows the allowed transitions in the case $N=8$; only
transitions that minimise the communication height are shown. \figref{fig:N16}
shows the case $N=16$. Note that in accordance with Lemma~\ref{lem:B0}, only
configurations with $p\leqs M$ interfaces appear in the two types $A_p$ (with
isolated particles and/or holes) and $A'_p$ (without isolated particles and/or
holes). 

The first-order correction $H^{(1)}$ to communication heights depends
not only on the number $M$ of particles, but also on the number $p$ of
interfaces. This is a nonlocal effect of the mass-conservation constraint.
However, in the limit $M\to\infty$, the four possible corrections converge
respectively to $\frac52$, $\frac12$, $-\frac12$ and $-\frac32$, i.e.~they no
longer depend on $p$. 

With this information at hand, it is now possible to determine the metastable
hierarchy among the families $A_p$ and $A'_p$. The result, which is proved in
Appendix~\ref{ssec_proof_B0}, reads as follows.

\begin{theorem}[Metastable hierarchy of particle/hole configurations]
\label{thm:meta_B0}
Let $M'$ be the largest even number less or equal $M = \frac N2$. Then 
\begin{equation}
 \label{eq:B0_hierarchy}
 A_2 \prec A'_4 \prec A'_6 \prec \dots \prec A'_{M'-2} 
 \prec A'_{M'} 
 \prec A_4 \prec A_6 \prec \dots \prec A_{N-2} \prec A_N\;.
\end{equation} 
defines a metastable order of the families $A_p$ and $A'_p$.
\end{theorem}


\section{Analysis of the dynamics}
\label{sec_dynamics}


\subsection{Interface dynamics}
\label{ssec_dynamics_interface}

The transition rules and communication heights given in
Proposition~\ref{prop:transitions_B0} and the metastable hierarchy obtained in
Theorem~\ref{thm:meta_B0} yield complementary information on the dynamics
between particle/hole configurations in $B_0$. Recall that the process behaves
essentially as a Markovian jump process with transition rates of order
$\e^{-H(x^\star_i,x^\star_j)/\eps}$, while the
hierarchy~\eqref{eq:B0_hierarchy} classifies the states according to the time
the process spends in them in metastable equilibrium. 

At the bottom of the metastable hierarchy, we find the set $A_2$ of
configurations having one cluster of $M$ particles: this constitutes the ground
state of the system, which can be interpreted as a solid or condensed phase. At
the top of the hierarchy on $B_0$, we find the set $A_N$ of states with $N$
interfaces. These consist in $M$ isolated particles, and can be interpreted as a
gaseous phase. 

\begin{figure}[tb]
\begin{center}
\scalebox{0.9}{
\input{fig_N16}
}
\end{center}
\vspace{-3mm}
\caption[]{Minimal transitions between particle/hole configurations in $B_0$
in the case $N=16$. Arrows indicate transitions that decrease the energy.}
\label{fig:N16} 
\end{figure}
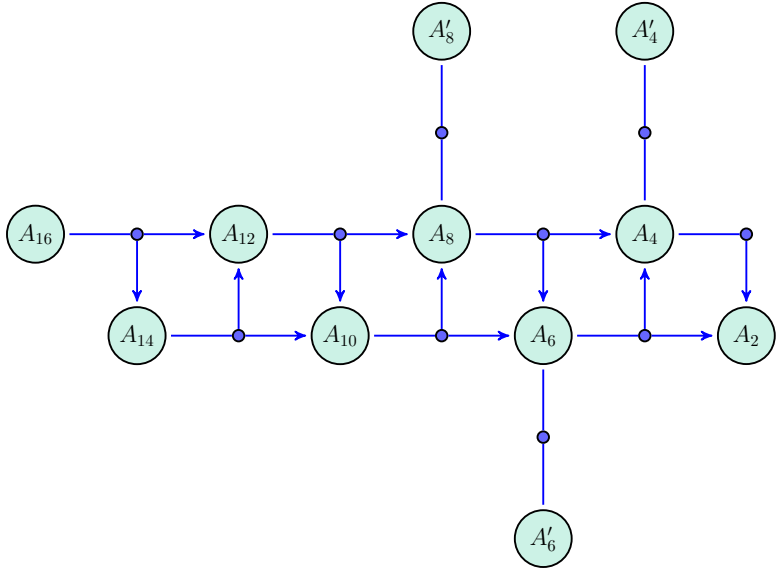

The transition graph implied by Proposition~\ref{prop:transitions_B0} (and
illustrated in Figures~\ref{fig:B0_N8} and~\ref{fig:N16}) shows that when
starting in the configuration $A_N$, the most likely transitions gradually
decrease the number $p$ of interfaces, in steps of $2$ or $4$. Thus the system
tends to gradually build clusters of increasing size. As the communication
heights given in Table~\ref{table:transitions_B0} increase as $p$ decreases,
this condensation process becomes slower as the size of clusters increases. This
is different from the usual Kawasaki dynamics, in which the transition rates
depend only on the change $\Delta p$ of the number of interfaces. Note in
particular that for given $p$, transitions of type \rom{4}, \rom{5} and \rom{6}
all occur at the same rate.

When the number $p$ of interfaces reaches $M$ (meaning that there are on average
$2$ particles per cluster), new configurations $A'_p$ become possible. These
consist of $p$ clusters separated by at least $2$ sites, and appear as dead ends
on the transition graph. The metastable order~\eqref{eq:B0_hierarchy} shows that
these configurations are actually more stable than those of type $A_p$,
$p\geqs4$, which have isolated particles or holes, and act as gateways to
configurations with fewer interfaces. In particular, configurations in $A'_4$
are those with the longest metastable lifetime. The system can spend
considerable time trapped in configurations with $p\geqs2$ clusters of
particles, separated by $p$ clusters of holes (as seen 
in~\figref{fig:interface_dynamics}). 


\subsection{Spectral gap}
\label{ssec_spectral gap}

Another interesting information on the process that can be obtained from its
metastable hierarchy is its spectral gap. We already know that the generator
$\cL$ admits the eigenvalue $0$, which is associated with the invariant
distribution~\eqref{eq:model09}. This eigenvalue is simple because the process
is irreducible and positive recurrent. The spectral gap is thus given by the
smallest nonzero eigenvalue $\lambda_2$ of $-\cL$, which governs the rate of
relaxation to equilibrium. 

At first glance, one might think that the spectral gap has order
$\e^{-(V_\gamma(z^\star)-V_\gamma(y^\star))/\eps}$, where $z^\star$ is a
$1$-saddle in $C_1$ and $y^\star$ is a local minimum in $B_1$. Indeed, this is
the inverse of the longest transition time obtained in Corollary~\ref{cor:Bk}. 
However, the corollary only applies to symmetric initial
distributions, and transitions from $B_1$ to $B_0$ via $C_1$ are not the slowest
processes of the system. In fact, this role is played by transitions between
configurations in $A_2$, which occur via saddles in $C_1$, leading to a
spectral gap of order $\e^{-(V_\gamma(z^\star)-V_\gamma(x^\star))/\eps}$, where
$x^\star$ is a local minimum in $A_2$ rather than $B_1$. Applying the
theory for symmetric processes in~\cite{BD15}, we obtain the following result. 
Its proof is given in Appendix~\ref{sec_proof_gap}.

\begin{theorem}[Spectral gap]
\label{thm:spectral_gap}
If $\eps$ is small enough, then the smallest nonzero eigenvalue of
$-\cL$ is given by 
\begin{equation}
 \label{eq:sg01}
 \lambda_2 
 = 4\sin^2 \biggpar{\frac{\pi}{N}}
 \frac{\abs{\lambda_-(z^\star)}}{2\pi} 
 \sqrt{\frac{\det\nabla^2V_\gamma(x^\star)}
 {\abs{\det\nabla^2V_\gamma(z^\star)}}}
\e^{-\brak{V_\gamma(z^\star)-V_\gamma(x^\star)}/\eps} 
\bigbrak{1 + \Order{\eps^{1/2}\abs{\log\eps}^{3/2}}}\;, 
\end{equation} 
where $x^\star$ is any configuration in $A_2$, and $z^\star$ is any saddle in
$C_1$ whose limit as $\gamma\to0$ has exactly $3$ interfaces. In particular, we
have 
\begin{align}
\nonumber
 V_\gamma(z^\star)-V_\gamma(x^\star)
 &= \frac{M(M-1)}{4(M^2-3M+3)} + \gamma \, \frac{M^2-6M+6}{2(M^2-3M+3)} 
 + \Order{\gamma^2}\\
 &= \frac14 + \frac12\gamma + \Order{N^{-1}} + \Order{\gamma^2}
 \;.
 \label{eq:sg02}
\end{align} 
Furthermore, 
\begin{align}
\nonumber
 \abs{\lambda_-(z^\star)}
 \sqrt{\frac{\det\nabla^2V_\gamma(x^\star)}
 {\abs{\det\nabla^2V_\gamma(z^\star)}}}
 &= \sqrt{2} \Biggbrak{\frac{M^2-3M+3}{(M-\frac32)\sqrt{M(M-3)}}}^{M-2} 
 + \Order{\gamma} \\
 &= \sqrt{2} + \Order{N^{-1}} + \Order{\gamma}\;.
 \label{eq:sg03}
\end{align} 
\end{theorem}

The fact that the spectral gap~\eqref{eq:sg01} decays like $N^{-2}$ for large
$N$ is highly nontrivial. It is related to the fact that the symmetry group 
$\mathfrak{D}_N\times\Z_2$ admits irreducible representations of dimension $2$,
and its computation requires the full power of the theory developed
in~\cite{BD15}. 

Physically, this result means that some transitions between states in $A_2$
require a time of order $N^2 \e^{1/4\eps}$, i.e., increasing as the square of
the system size when the noise intensity $\eps$ is constant. In other words,
the motion of interfaces slows down like $N^{-2}$ when the system becomes
large.


\section{Conclusion}
\label{sec_conclusion}

Let us briefly summarise the main results obtained in this work.
\begin{itemiz}
\item 	Using the concept of metastable hierarchy, the long-term dynamics of
the system can be reduced to an effective process jumping between particle/hole
configurations. These configurations exist as long a the coupling intensity
$\gamma$ is smaller than a critical value, bounded below by a
constant independent of the system size.
\item 	The effective dynamics tends to reduce the number of interfaces, and
slows down as this number decreases. As soon as the average size of clusters
reaches $2$, the system can get trapped in configurations without isolated
sites, which are more stable than any configuration with isolated sites. 
\item 	The spectral gap is of order $N^{-2} \e^{-1/4\eps}$, which decreases as
the square of the inverse of the system size. This means that transitions
between the $N$ configurations forming the ground state $A_2$ slow down as $N$
increases.
\end{itemiz}

We emphasise that all results obtained here apply for arbitrarily large but
finite system size $N$. In fact, some quantities like the number $\theta$
defining the metastable hierarchy go to zero in the limit $N\to\infty$, so that
the orders~\eqref{eq:Bk01} and~\eqref{eq:B0_hierarchy} only make sense for
finite~$N$. We do not claim either that the error terms of order
$\eps^{1/2}\abs{\log\eps}^{3/2}$ in~\eqref{eq:sg01} and~\eqref{eq:Bk02} are
uniform in $N$, though results obtained in a similar situation in~\cite{BBM10}
indicate that they probably are.  

A different situation of interest, not considered here, arises when the
coupling intensity $\gamma$ grows like $N^2$. Then one expects that the system
converges to a mass-conserving Allen--Cahn SPDE on a bounded interval, which has
considerably fewer metastable states. Indeed, an analogous scenario was
obtained in~\cite{BFG06b}, where the unconstrained system with $\gamma\sim N^2$
was shown to have only $2$ local minima, and at most $2N$ saddles of index $1$. 
If, by contrast, one has $1 \ll \gamma \ll N^2$, a scaling argument shows that
the system should converge to an Allen--Cahn SPDE on a growing domain, which
admits more metastable states; see in particular~\cite{VdE-W_08,OWW_14} for
results in the unconstrained case, and~\cite{Mourrat_Weber_14} for a recent
convergence result in dimension $2$.

The behaviour of the constrained system for lattices of dimension larger than
$1$ remains so far an open problem. The phenomenology is expected to be
different, because the energy of clusters then depends not only on the size of
their interfaces, but also on the size of their bulk. This can result in
scenarios where the interface dynamics accelerates once a critical droplet size
has been reached, as is well known for lattice systems with standard Kawasaki
dynamics~\cite{denHollander04}. 


\appendix

\section{Proofs: Potential landscape}
\label{sec_proof_landscape}


\subsection{The uncoupled case}
\label{ssec_proof_landscape_0}

\begin{proof}[Proof of Proposition~\ref{prop_saddles}]
Consider a critical point $x^\star$ of the constrained system with triple
$(a_0,a_1,a_2)$. Recall that this means that $x^\star$ has $a_j$ coordinates
equal to $\alpha_j$, $j=0,1,2$, where the $\alpha_j$ are distinct roots of
$\xi^3-\xi-\lambda$ for some $\lambda\in(-\lambdac,\lambdac)$. By Vieta's
formula, these roots satisfy 
\begin{equation}
 \label{eq:pl0_01}
 \alpha_0 + \alpha_1 + \alpha_2 = 0\;.
\end{equation}
We always have $a_0+a_1+a_2=N$, and by convention $a_0\leqs a_1\leqs a_2$.
Note that we may assume $a_0\neq a_2$, since otherwise all $a_j$ would be
equal, and thus $N$ would be a multiple of $3$, which is excluded by
assumption. 

Combining~\eqref{eq:pl0_01} with the constraint $\sum x^\star_i = a_0\alpha_0 +
a_1\alpha_1 + a_2\alpha_2 = 0$ yields the
relation 
\begin{equation}
 \label{eq:pl0_02}
 (a_1-a_0)\alpha_1 + (a_2-a_0)\alpha_2 = 0\;.
\end{equation} 
Solving for $\alpha_2$ and using the fact that all $\alpha_j^3-\alpha_j$ are
equal, a short computation shows that 
\begin{align}
\nonumber
\alpha_0 &= \pm(a_1-a_2) R^{1/2}\;,
\\
\alpha_1 &= \pm(a_2-a_0) R^{1/2}\;,
\label{eq:alpha_j} 
\\
\alpha_2 &= \pm(a_0-a_1) R^{1/2}\;,
\nonumber
\end{align}
where
\begin{equation}
 \label{eq:R12}
 R = \frac{a_2+a_1-2a_0}{(a_1-a_0)^3+(a_2-a_0)^3}
 = \frac1{a_0^2+a_1^2+a_2^2-a_0a_1-a_0a_2-a_1a_2}\;.
\end{equation} 
We now turn to determining the signature of the Hessian at these critical points
of the potential $V_\gamma$ restricted to the hyperplane $S$. This signature
does not depend on the parametrisation of $S$, so that it is equal to the
signature of the Hessian of 
\begin{equation}
 \label{eq:pl0_03}
 \widetilde V_\gamma(x_1,\dots,x_{N-1})
 = V_\gamma(x_1,\dots,x_{N-1},-x_1-\dots-x_{N-1})\;.
\end{equation} 
Computing the Hessian of $\widetilde V_\gamma$ at $x^\star$ shows that it has
the form 
\begin{equation}
 \label{eq:pl0_04}
 H = 
 \begin{pmatrix}
  (3\alpha_0^2-1)\one_{a_0} & 0 & 0 \\
  0 & (3\alpha_1^2-1)\one_{a_1} & 0 \\
  0 & 0 & (3\alpha_2^2-1)\one_{a_2-1}
 \end{pmatrix}
 + (3\alpha_2^2-1) 
 \begin{pmatrix}
 1 & \cdots & 1 \\
 \vdots & \ddots & \vdots \\
 1 & \cdots & 1
 \end{pmatrix}\;,
\end{equation} 
where $\one_a$ denotes the identity matrix of size $a$. We now distinguish
between the following cases. 
\begin{enum}
\item 	$(a_0,a_1,a_2)=(0,0,N)$. Then $x^\star=0$, and one easily sees that $-H$
is positive definite, so that $x^\star$ is a saddle of index $N-1$. 

\item 	$a_0=0$ and $a_1\geqs1$. Using the expressions~\eqref{eq:alpha_j}, we
obtain that $3\alpha_1^2-1>0$ and $3\alpha_2^2-1$ has the same sign as
$2a_1-a_2$. Let $e_1, \dots, e_{N-1}$ denote the canonical basis vectors. Then
$\set{e_1-e_i}_{i\in\intint{2}{a_1}}$ are eigenvectors of $H$ with eigenvalue
$3\alpha_1^2-1$, and $\set{e_{a_1+1}-e_i}_{i\in\intint{a_1+2}{N-1}}$ are
eigenvectors of $H$ with eigenvalue $3\alpha_2^2-1$. 

To find the remaining two eigenvalues, let $u=\sum_{i=1}^{a_1}e_i$ and
$v=\sum_{i=a_1+1}^{N-1}e_i$. These two vectors span an invariant subspace of
$H$, in which the action of $H$ takes the form 
\begin{equation}
\label{eq:pl0_05}
\cM =
\begin{pmatrix}
3\alpha_1^2-1 + a_1(3\alpha_2^2-1) & (a_2-1)(3\alpha_2^2-1) \\
a_1(3\alpha_2^2-1) & a_2(3\alpha_2^2-1)
\end{pmatrix}\;.
\end{equation}
Computing the determinant and the trace of $\cM$, one sees that if $2a_1>a_2$,
then the two eigenvalues of $\cM$ are strictly positive, so that $x^\star$ is a
stationary point of index 0. If $2a_1<a_2$, then $\cM$
has one strictly positive and one strictly negative eigenvalue, and $x^\star$
has index $a_2-1$. 

\item 	$a_0\geqs1$. In that case one finds that $3\alpha_1^2-1>0$, while
$3\alpha_0^2-1$ has the same sign as $(2a_2-a_1-a_0)(a_0-2a_1+a_2)$ and 
$3\alpha_2^2-1$ has the same sign as $(2a_0-a_1-a_2)(a_0-2a_1+a_2)$. 
Here it is better to invert the r\^oles of $\alpha_1$ and $\alpha_2$ in the
expression for $H$. Similarly to the previous case, one finds 
$a_0-1$ eigenvectors with eigenvalue $3\alpha_0^2-1$,  
$a_2-1$ eigenvectors with eigenvalue $3\alpha_2^2-1$ and   
$a_1-2$ eigenvectors with eigenvalue $3\alpha_1^2-1$ (these eigenvectors are of
the form $e_1-e_i$, $e_{a_0+1}-e_i$ and $e_{a_0+a_2+1}-e_i$ for appropriate
ranges of $i$). 

To find the other eigenvalues, let $u=\sum_{i=1}^{a_0}e_i$, 
$v=\sum_{i=a_0+1}^{a_0+a_2}e_i$ and $w=\sum_{i=a_0+a_2+1}^{N-1}e_i$. 
These span an $H$-invariant subspace, in which the action of $H$ takes the form 
\begin{equation}
\label{eq:pl0_06}
\cM =
\begin{pmatrix}
3\alpha_0^2-1 + a_0(3\alpha_1^2-1) & a_2(3\alpha_1^2-1) 
& (a_1-1)(3\alpha_1^2-1) \\
a_0(3\alpha_1^2-1) & 3\alpha_2^2-1+a_2(3\alpha_1^2-1)
& (a_1-1)(3\alpha_1^2-1) \\
a_0(3\alpha_1^2-1) & a_2(3\alpha_1^2-1) & a_1(3\alpha_1^2-1) 
\end{pmatrix}\;.
\end{equation}
In this case, one finds $\Tr \cM>0$, and $\det \cM$ has the same sign as
$a_0-2a_1+a_2$. If $\det \cM<0$, then $\cM$ has two strictly positive and one
strictly negative eigenvalue, and $x^\star$ has index $a_0$. If $\det \cM >0$,
computing the term of degree $1$ of the characteristic polynomial of $\cM$ one
concludes that all eigenvalues of $\cM$ are strictly positive, and that
$x^\star$ has index $a_2-1$. 
\qed
\end{enum}
\renewcommand{\qed}{}
\end{proof}

\begin{proof}[Proof of Theorem~\ref{thm:transition_graph}]
Let $z^\star\in C_k$ be a $1$-saddle. Its triple can be written $(1,a-1,N-a)$
where $a=\frac N2-k+1 \in\intint{\frac N2+1-k_{\max}}{\frac N2}$. We shall
construct a path $\Gamma$, connecting $z^\star$ to a point $x^\star\in B_{k-1}$
of triple $(0,a,N-a)$, and such that the potential $V_0$ is decreasing along
$\Gamma$. An analogous construction holds for the connection from $z^\star$ to
a local minimum in $B_k$. 

In fact it will turn out to be sufficient to use a linear path. Reordering the
components if necessary, we may assume that
$x^\star=(\alpha_1,\dots,\alpha_1,\alpha_2,\dots,\alpha_2)$ with $\alpha_1$
repeated $a$ times and $\alpha_2$ repeated $N-a$ times, and 
$z^\star=(\alpha'_0,\alpha'_1,\dots,\alpha'_1,\alpha'_2,\dots,\alpha'_2)$, with
$\alpha'_1$ repeated $a-1$ times and $\alpha'_2$ repeated $N-a$ times. Note
that these points indeed satisfy the connection rules~\eqref{eq:gamma0_rules}. 
Let $\Gamma(t)=tz^\star + (1-t)x^\star$ and set $h(t)=V_0(\Gamma(t))$. Then a
direct computation shows that 
\begin{align}
\nonumber
h'(t) ={}& 
(\alpha_0'-\alpha_1) \bigbrak{((1-t)\alpha_1+t\alpha'_0)^3 
- ((1-t)\alpha_1+t\alpha_0')} \\
\nonumber
&{}+ (a-1)(\alpha_1'-\alpha_1) \bigbrak{((1-t)\alpha_1+t\alpha'_1)^3 
- ((1-t)\alpha_1+t\alpha_1')} \\
&{}+ (N-a)(\alpha_2'-\alpha_2) \bigbrak{((1-t)\alpha_2+t\alpha'_2)^3 
- ((1-t)\alpha_2+t\alpha_2')}\;.
\label{eq:pl0_07} 
\end{align}
The properties of the $\alpha_j$ and $\alpha'_j$ yield $h'(0)=h'(1)=0$.
Since $h'(t)$ is a polynomial of degree $3$, it can be written as 
\begin{equation}
 \label{eq:lp0_08}
 h'(t) = Kt(t-1)(t-\psi)
\end{equation} 
for some $K, \psi\in\R$. Computing the coefficient of $t^3$ in~\eqref{eq:pl0_07}
yields $K>0$. Thus if we manage to show that $\psi>1$, we can indeed conclude
that $h'(t)<0$ on $(0,1)$, showing that $h(t)$ is decreasing as required. The
condition $\psi>1$ is equivalent to having $h''(1)<0$. Using the
expressions~\eqref{eq:alpha_j} of the $\alpha_j$, one obtains after some algebra
that 
\begin{align}
\nonumber
h''(1) ={}&
2(\omega')^2 \bigbrak{(a-1)(9a-8N)+aN^2-a^2N} 
-4\omega\omega'N(N-a)(a-2) -\omega^2aN(N-a) \\
&{}+ 3(\omega\omega')^2(N-a)
\bigbrak{aN^3-3a^2N^2+3Na^2 - (a-1)(9a^2-9aN+4N^2)}\;,
\label{eq:lp0_09} 
\end{align}
where $\omega=(N^2-3aN+3a^2)^{-1/2}$ and 
$\omega'=(N^2-3aN+3(a^2-a+1))^{-1/2}$ stem from the terms $R^{1/2}$ 
in~\eqref{eq:alpha_j}. Using the fact that $\omega\omega'N(N-a)(a-2)>0$,
rearranging and replacing $\omega$ and $\omega'$ by their values, the condition
$h''(1)<0$ can be seen to be true if the condition $g(a)<0$ holds, where 
\begin{align}
\nonumber
g(a) ={}&
9(8N-27)a^4 + -3(56N^2-156N-81)a^3 +3N(48N^2-101N-156)a^2\\
&{}- N^2(56N^2-74N-303)a + 2N^3(4N^2-37)\;.
\label{eq:lp0_10} 
\end{align} 
To check the condition, first observe that if $N\geqs4$ then $g^{(4)}(a)>0$ for
all $a$. Next check that $g^{(3)}(\frac N2)<0$ for $N\geqs4$ to conclude that
$g^{(3)}(a)<0$ for all $a\leqs\frac N2$. Proceeding in a similar way with the
second and first derivatives of $g$, one reaches the conclusion that $g(a)$ is
decreasing for $a\leqs\frac N2$ if $N\geqs4$. It thus remains to show that $g$
is negative at the left boundary of its domain of definition. This follows by
checking the slightly stronger condition $g(\frac N3+\frac43)<0$. 
\end{proof}


\subsection{The case of small positive coupling}
\label{ssec_proof_landscape_gamma}

To prove Theorem~\ref{thm:persistence}, we proceed in two steps. First we
ignore the constraint that stationary points $x^\star$ should belong to the
hyperplane $S$, and prove that the equation 
\begin{equation}
 \label{eq:plg_01}
 \nabla V_\gamma(x) = \lambda\vone
\end{equation} 
admits exactly $3^N$ solutions for all $(\gamma,\lambda)$ in a given domain. 
Then we obtain conditions on $(\gamma,\lambda)$ guaranteeing that these
stationary points belong to $S$. 

\begin{figure}[tb]
\begin{center}
\scalebox{1}{
\input{fig_D2}
}
\end{center}
\vspace{-3mm}
\caption[]{The domain $D$ in the $(\gamma,\lambda)$-plane defined
in~\eqref{eq:def_D} (the boundaries of $D$ are not straight line segments,
although they look straight). For all $(\gamma,\lambda)\in D$, the equation
$\nabla V_\gamma(x)=\lambda\vone$ admits $3^N$ stationary points. The smaller
domain corresponds to the parameter values where stationary points of the
family $B_0$ can exist in the hyperplane $S$.}
\label{fig:D} 
\end{figure}
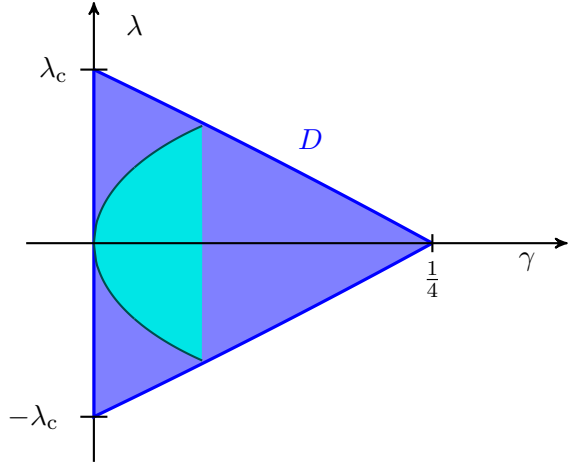

Let $\lambdac = \frac{2}{3\sqrt{3}}$ and define  
\begin{equation}
 \label{eq:def_D}
D = \bigsetsuch{(\gamma,\lambda)\in[0,\tfrac14]\times[-\lambdac,\lambdac]}
{\abs{\lambda} + \gamma \hat \alpha(\lambda) \leqs \lambdac(1-\gamma)^{3/2}}\;,
\end{equation}
where $\hat \alpha(\lambda)$ is the largest root of $x^3 - x - \abs{\lambda}$. 
The set $D$ is shown in \figref{fig:D}. A simpler sufficient condition for being
in $D$ is obtained by observing that 
\begin{equation}
 \label{eq:def_Dprime}
 D \supset D' =
\bigsetsuch{(\gamma,\lambda)\in[0,\tfrac29]\times[-\lambdac,\lambdac]}
{\abs{\lambda} \leqs \lambdac (1-\tfrac92\gamma)}\;,
\end{equation} 
owing to the fact that $\hat \alpha(\lambda)\in[1,\frac{2}{\sqrt3}]$
for $\abs{\lambda}\leqs\lambdac$.

\begin{prop}
\label{prop:horseshoe1} 
If $(\gamma,\lambda)\in D$, then \eqref{eq:plg_01} admits exactly
$3^N$ solutions, depending continuously on $\gamma$ and $\lambda$.
\end{prop}
\begin{proof}
The proof, in the spirit of~\cite{Keener87}, is based on the construction of a
horseshoe-type map admitting an invariant Cantor set on which the dynamics is
conjugated to the full shift on $3$ symbols. 
First note that we may assume $0 < \gamma \leqs\frac14$, the case $\gamma=0$
having already been dealt with. Let $f_\lambda(x)=x-x^3+\lambda$ and consider
the map $T:\R^2\to\R^2$ given by 
\begin{equation}
 \label{eq:plg_03}
 T(x,y) = \Bigpar{2x-y-\frac2\gamma f_\lambda(x), x}\;.
\end{equation} 
This is an invertible map, with inverse $T^{-1}=\Pi \circ T \circ \Pi$ where
$\Pi$ is the involution given by $\Pi(x,y)=(y,x)$. Furthermore, the relation
$T(x_n,x_{n-1})=(x_{n+1},x_n)$ is equivalent to 
\begin{equation}
 \label{eq:plg_04}
 x_n^3 - x_n - \frac\gamma2 \bigpar{x_{n+1}-2x_n+x_{n-1}} = \lambda\;.
\end{equation} 
This shows that fixed points of $T^N$ are in one-to-one
correspondence with solutions of~\eqref{eq:plg_01}. 
Our aim is thus to show that when $(\gamma,\lambda)\in D$, the map $T$ has
exactly $3^N$ periodic orbits of (not necessarily minimal) period $N$. To this
end, we construct some subsets of $\R^2$ which behave nicely under the map
$T$. 

\begin{figure}[tb]
\begin{center}
\scalebox{1}{
\input{fig_horseshoe}
\hspace{5mm}
\input{fig_horseshoe_H}
}
\end{center}
\vspace{-3mm}
\caption[]{The sets $\cV_\sigma$ and $\cH_\sigma$ constructed in the proof of
Proposition~\ref{prop:horseshoe1}. The square is the set
$[\alphamin,\alphamax]^2$. The $\cV_\sigma$ are bounded below by
$g(x)-\alphamax$
and above by $g(x)-\alphamin$. Each $\cV_\sigma$ is mapped by $T$ to the
corresponding $\cH_\sigma$. Iterating $T$ forward and backward in time produces
an invariant Cantor set contained in the intersections of the $\cV_\sigma$ and
$\cH_\sigma$.}
\label{fig:horseshoe} 
\end{figure}
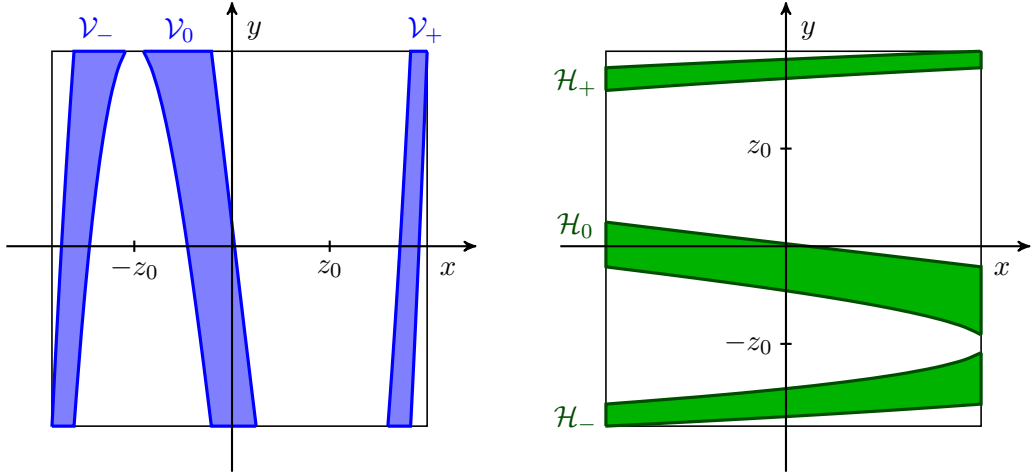

We can write $T(x,y)=(g(x)-y,x)$ where $g$ is the function 
\begin{equation}
 \label{eq:plg_05}
 g(x) = 2x - \frac2\gamma f_\lambda(x)
 = \frac2\gamma \bigbrak{x^3 - (1-\gamma)x - \lambda}\;.
\end{equation} 
It has a local minimum at $z_0=\sqrt{(1-\gamma)/3}$ and a local maximum at
$-z_0$. Furthermore, it is strictly increasing on $(-\infty,-z_0)$ and
$(z_0,\infty)$ and strictly decreasing on $(-z_0,z_0)$. 
Let $\alphamin$ and $\alphamax$ be the smallest and largest roots
of $x^3-x-\lambda$. Note that $\max\set{\alphamax,-\alphamin} = \hat \alpha$ and
that 
\begin{equation}
 \label{eq:plg_05c}
 g(\alphamax) = 2\alphamax\;, \qquad
  g(\alphamin) = 2\alphamin\;.
\end{equation}
Furthermore one can check that 
\begin{equation}
 \label{eq:plg_06}
 (\gamma,\lambda) \in D 
 \quad\Rightarrow\quad
 g(-z_0) \geqs 2\alphamax \quad\text{and}\quad g(z_0) \leqs 2\alphamin\;.
\end{equation} 
Denote by $g_-^{-1}$ the inverse of $g$ with range $[\alphamin,-z_0]$ and
introduce the \lq\lq vertical\rq\rq\ strip  
\begin{equation}
 \label{eq:plg_07}
 \cV_- = 
 \bigsetsuch{(x,y)}{g_-^{-1}(y+\alphamin)\leqs x\leqs g_-^{-1}(y+\alphamax), 
 \alphamin\leqs y\leqs \alphamax}
\end{equation} 
(see \figref{fig:horseshoe}). 
Then we see that $T$ maps $\cV_-$ to the \lq\lq horizontal\rq\rq\ strip 
$\cH_- = \Pi \cV_-$. Similarly, if $g_0^{-1}$ denotes the inverse of $g$ with
range $[-z_0,z_0]$, then the strip 
\begin{equation}
 \label{eq:plg_08}
 \cV_0 = 
 \bigsetsuch{(x,y)}{g_0^{-1}(y+\alphamax)\leqs x\leqs g_0^{-1}(y+\alphamin), 
 \alphamin\leqs y\leqs \alphamax}
\end{equation} 
is mapped by $T$ to $\cH_0 = \Pi \cV_0$. In the same way, one can construct a
strip $\cV_+$ defined via the inverse $g_+^{-1}$ of $g$ with range
$[z_0,\alphamax]$, which is mapped to $\cH_+ = \Pi\cV_+$. The
property~\eqref{eq:plg_06} ensures that the strips $\cV_\sigma$ have disjoint
interiors, and the same holds for the $\cH_\sigma$. 

Consider now any finite word $\omega=(\omega_{-n},\dots,\omega_{n+1})
\in\set{-,0,+}^{2(n+1)}$, and associate with it the set 
\begin{equation}
 \label{eq:plg_09}
 I_\omega = \bigcap_{k=-n}^{n+1} T^k(\cV_{\omega_k})\;.
\end{equation} 
The above properties of the strips imply that all $I_\omega$ are non-empty, and
have pairwise disjoint interior. In fact, the union of all $I_\omega$ converges
as $n\to\infty$ to a Cantor set invariant under $T$. By a standard
argument~\cite{Keener87}, for every doubly infinite sequence
$\omega\in\set{-,0,+}^\Z$, there exists an $I_\omega\in[\alphamin,\alphamax]^2$
whose orbit visits $\cV_{\omega_n}$ at time $-n$ and $\cH_{\omega_n}$ at time
$n+1$ for each $n\in\N_0$. In particular, for any of the $3^N$ possible
$N$-periodic sequences $\omega$, we obtain exactly one $N$-periodic orbit of
$T$, which corresponds to one solution of~\eqref{eq:plg_01}. It depends
continuously on the parameters $\gamma$ and $\lambda$, because the $I_\omega$
depend continuously on them.  
\end{proof}

Let us point out that the above result is consistent with the previously
obtained properties of the system for $\gamma=0$. Indeed, as $\gamma\to0$, the
function $g$ defined in~\eqref{eq:plg_05} becomes singular, switching between
$-\infty$ and $+\infty$ at the roots of $x^3-x-\lambda$, which are precisely the
$\alpha_j$ introduced in Section~\ref{ssec_proof_landscape_0}. As a consequence,
the invariant Cantor set collapses on $\set{\alpha_0,\alpha_1,\alpha_2}^2$, and
the stationary points are all $N$-tuples with these coordinates (there are
indeed $3^N$ of them). 

In order to deal with the constraint $x^\star\in S$, we will need some control
on the size of the sets $\cV_\sigma \cap \cH_{\sigma'}$. The following lemma
provides upper bounds on the widths of the $\cV_\sigma$ (and thus also on the
heights of the $\cH_{\sigma'}$) which will be sufficient for this purpose. 

\begin{lemma}
\label{lem:horseshoe}
Assume that $(\gamma,\lambda)\in D$, and denote by $\alphamin < \alphac <
\alphamax$ the three roots of $x^3-x-\lambda$. Then
\begin{align}
\nonumber
\cV_- &\subset 
\bigbrak{\alphamin,\alphamin + \sqrt{\gamma}\,}
\times \bigbrak{\alphamin,\alphamax}\;, \\
\nonumber
\cV_0 &\subset \bigbrak{\alphac - \sqrt{\gamma}, \alphac + \sqrt{\gamma}\,}
\times \bigbrak{\alphamin,\alphamax}\;,\\
\cV_+ &\subset \bigbrak{\alphamax - \sqrt{\gamma},\alphamax}
\times \bigbrak{\alphamin,\alphamax}\;.
 \label{eq:plg_20}
\end{align} 
\end{lemma}
\begin{proof}
Denote by $x_1$ the $x$-coordinate of the top-right corner of $\cV_-$
(see \figref{fig:horseshoe}). Then we have the relations  
\begin{align}
\nonumber
x_1^3 - (1-\gamma)x_1 - \lambda &= \gamma \alphamax\;, \\
\alphamin^3 - (1-\gamma)\alphamin - \lambda &= \gamma \alphamin\;. 
\label{eq:plg22} 
\end{align}
Taking the difference of the two lines, writing $x_1=\alphamin+\Delta_1$ and
recalling the definition $z_0=\sqrt{(1-\gamma)/3}$ yields
\begin{equation}
 \label{eq:plg23}
 \Delta_1 h_1(\Delta_1) = \gamma(\alphamax-\alphamin)\;, \qquad 
h_1(\Delta) = 3(\alphamin^2-z_0^2) +3\alphamin\Delta + \Delta^2\;.
\end{equation}
One easily checks that the map $\Delta\mapsto h_1(\Delta)/\Delta$ is decreasing.
Since $x_1 \leqs -z_0$ and thus $\Delta_1\leqs -z_0-\alphamin$, it follows that 
\begin{equation}
 \label{eq:plg24}
 \frac{h_1(\Delta_1)}{\Delta_1} \geqs
\frac{h_1(-z_0-\alphamin)}{-z_0-\alphamin} 
 = 2z_0 - \alphamin\;.
\end{equation} 
As a consequence, $(2z_0-\alphamin) \Delta_1^2 \leqs
\gamma(\alphamax-\alphamin)$, so that we conclude that 
\begin{equation}
\cV_- \subset 
\biggbrak{\alphamin,\alphamin +
\biggpar{\frac{\gamma(\alphamax-\alphamin)}{2z_0-\alphamin}}^{1/2}\,}
\times \bigbrak{\alphamin,\alphamax}\;. 
 \label{eq:plg_20A}
\end{equation} 
Now we claim that $\alphamax \leqs 2z_0$ holds for all $(\gamma,\lambda)\in D$. 
Indeed, if $g$ is the function defined in~\eqref{eq:plg_05}, then we have
by~\eqref{eq:def_D}
\begin{equation}
 \label{eq:plg25}
 g(2z_0) = \frac{2}{\gamma} \bigbrak{\lambdac (1-\gamma)^{3/2} - \lambda} 
 \geqs 2 \hat\alpha(\lambda)
\end{equation} 
for all $(\gamma,\lambda)\in D$. Hence by~\eqref{eq:plg_05c} we get
$g(2z_0) \geqs 2\alphamax = g(\alphamax)$, showing as claimed that $\alphamax
\leqs 2z_0$ since $g$ is increasing on $[z_0,\infty)$. Using this bound
in~\eqref{eq:plg_20A} yields the first relation in~\eqref{eq:plg_20}. 

In a similar way, if $x_2$ denotes the $x$-coordinate of the top-left corner of
$\cV_0$, one obtains that $\Delta_2=\alpha_c-x_2$ satisfies 
\begin{equation}
 \label{eq:plg26}
 \Delta_2 h_2(\Delta_2) = \gamma(\alphamax-\alphac)\;, \qquad 
h_2(\Delta) = 3(z_0^2-\alphac^2) +3\alphac\Delta - \Delta^2\;.
\end{equation}
One obtains again that $\Delta\mapsto h_2(\Delta)/\Delta$ is decreasing, and
its smallest value, reached at $\Delta_2=\alphac+z_0$, is equal to
$2z_0-\alphac$. The other relevant coordinates can be computed in the same way,
yielding 
\begin{align}
\nonumber
\cV_0 &\subset \biggbrak{\alphac -
\biggpar{\frac{\gamma(\alphamax-\alphac)}{2z_0-\alphac}}^{1/2},\alphac +
\biggpar{\frac{\gamma(\alphac-\alphamin)}{2z_0+\alphac}}^{1/2}\,}
\times \bigbrak{\alphamin,\alphamax}\;,\\
\cV_+ &\subset \biggbrak{\alphamax - 
\biggpar{\frac{\gamma(\alphamax-\alphamin)}{2z_0+\alphamax}}^{1/2},\alphamax}
\times \bigbrak{\alphamin,\alphamax}\;.
 \label{eq:plg_20B}
\end{align} 
The conclusion follows as before using $\alphamax \leqs 2z_0$ and the symmetric
relation $-\alphamin \leqs 2z_0$.
\end{proof}

Fix a triple $(a_0,a_1,a_2)$, with as usual the $a_i$ increasing integers
of sum $N$. We denote by $\lambda_0$ the common value of the
$\alpha_j^3-\alpha_j$, where $\set{\alpha_j}_{j\in\set{0,1,2}}$ are 
given in~\eqref{eq:alpha_j}. 
For arbitrary $\lambda\in[-\lambdac,\lambdac]$ we define the quantity 
\begin{equation}
 \label{eq:plg_30}
 \Sigma_0(\lambda) = \frac{1}{N} \bigbrak{a_0\alpha_0(\lambda) +
a_1\alpha_1(\lambda) + a_2\alpha_2(\lambda)}\;,
\end{equation} 
where the $\alpha_j(\lambda)$ are three distinct roots of $x^3-x-\lambda$,
numbered in such a way that $\alpha_j(\lambda_0)=\alpha_j$. By construction, we
have $\Sigma_0(\lambda_0)=0$. 

Proposition~\ref{prop:horseshoe1} ensures the existence, for
$(\gamma,\lambda)\in D$, of a continuous family $x^\star(\gamma,\lambda)$ of
solutions of~\eqref{eq:plg_01}, such that $x^\star(0,\lambda)$ has $a_j$
coordinates equal to $\alpha_j(\lambda)$. We set 
\begin{equation}
 \label{eq:plg_40}
 \Sigma_\gamma(\lambda) = \frac{1}{N} \sum_{i=1}^N x^\star_i(\gamma,\lambda)\;.
\end{equation} 
It follows directly from Lemma~\ref{lem:horseshoe} that 
\begin{equation}
 \label{eq:plg_41}
 \Sigma_0(\lambda) - \sqrt{\gamma} \leqs \Sigma_\gamma(\lambda) 
 \leqs \Sigma_0(\lambda) + \sqrt{\gamma}\;.
\end{equation} 
If $\lambda\mapsto \Sigma_\gamma(\lambda)$ changes sign in $D$ at some
$\lambda_*(\gamma)$, then $x^\star(\gamma,\lambda_*(\gamma))$ is indeed a
stationary point of $V_\gamma$ satisfying the constraint $x^\star\in S$.
Assuming for the moment that such a point exists, the following result
characterises its signature.  

\begin{lemma}
\label{lem:horseshoe2}
Assume that $(\gamma,\lambda)\in \Int D'$, where $D'\subset D$ is defined
in~\eqref{eq:def_Dprime}. Then any stationary point $x^\star$ of the family
$B_k$, with triple $(a_0,a_1,a_2)=(0,M-k,M+k)$, is a local minimum of the
constrained potential $V_\gamma$. Furthermore, there exists a constant $c_0>0$
such that if $\gamma\leqs c_0\sqrt{\lambdac - \abs{\lambda}}$, then any
stationary point $x^\star$ of the family $C_k$, with triple
$(a_0,a_1,a_2)=(1,M-k-1,M+k)$, is a saddle of index $1$ of the constrained
potential $V_\gamma$.
\end{lemma}
\begin{proof}
First we note that by definition of $D'$, the function $g$ defined
in~\eqref{eq:plg_05} satisfies 
\begin{equation}
 \label{eq:plg_401}
 g\biggpar{-\frac{1}{\sqrt3}} = \frac2\gamma(\lambdac-\lambda) -
\frac{2}{\sqrt3} > \frac{4}{\sqrt3} \geqs 2\alphamax\;,
\end{equation} 
which implies that points in $\cV_-$ have a first coordinate $x$ satisfying
$x<-1/\sqrt3$, and thus $3x^2>1$. By symmetry, points in $\cV_+$ also have a
first coordinate satisfying $3x^2>1$. Stationary points $x^\star$ in the family
$B_k$ have all coordinates in $\cV_\pm$,  since they are deformations of points
with all coordinates equal to $\alphamin$ or $\alphamax$. The Hessian matrix
$H_\gamma$ of the unconstrained potential at any stationary point $x^\star$
defines the quadratic form 
\begin{equation}
 \label{eq:plg_403}
 v \mapsto \pscal{v}{H_\gamma v} 
 = \sum_{i=1}^N \bigpar{3(x^\star)^2 -1} v_i^2 
 + \frac\gamma2 \sum_{i=1}^N \bigpar{v_i-v_{i+1}}^2 \;.
\end{equation} 
This form is clearly positive definite if $x^\star\in B_k$, showing that
$x^\star$ is a local minimum of the unconstrained potential. Thus it is also a
local minimum of the constrained potential. 

In the case where $x^\star\in C_k$, it has exactly one coordinate $x$ in
$\cV_0$, for which one easily checks that $3x^2 < 1$. Thus $H_0$ has exactly
one negative eigenvalue, showing that for $\gamma=0$, $x^\star$ is a $1$-saddle
of the unconstrained potential. By Proposition~\ref{prop_saddles}, $x^\star$ is
also a $1$-saddle of the constrained potential, so that there exists a vector
$v\in S$ such that $\pscal{v}{H_0v} < 0$. In fact, one can deduce
from~\eqref{eq:pl0_06} that the negative eigenvalue of $H_0$ is bounded above by
$-c_1\sqrt{\lambdac - \abs{\lambda}}$ for a $c_1>0$, while its other eigenvalues
are bounded below by $c_1\sqrt{\lambdac - \abs{\lambda}}$. 
Since the second term in~\eqref{eq:plg_403} has an $\ell^2$-operator norm equal
to $\gamma$ (it is a discrete Laplacian, diagonalisable by discrete Fourier
transform), the Bauer--Fike theorem shows that $x^\star$ remains a $1$-saddle
of the unconstrained potential as long as $\gamma < c_2\sqrt{\lambdac -
\abs{\lambda}}$ for some $c_2>0$. 

To show that this also holds for the constrained potential, we can use the fact
that the eigenvectors of a perturbed matrix move by an amount controlled by the
size of the perturbation (see for instance~\cite[Thm.~4.1]{Deif_1995}). In this
way, we obtain the existence of an orthogonal matrix $O_\gamma$ such that
$\delta O_\gamma = O_\gamma-\one$ has order $\gamma/\sqrt{\lambdac -
\abs{\lambda}}$ and $D_\gamma=O_\gamma H_\gamma\transpose{O_\gamma}$ is
diagonal, with the same eigenvalues as $H_\gamma$. It follows by
Cauchy--Schwarz that 
\begin{align}
\nonumber
 \pscal{v}{H_\gamma v} 
 &= \pscal{O_\gamma v}{D_\gamma O_\gamma v} \\
\nonumber
 &= \pscal{v}{D_\gamma v} + 2\pscal{\delta O_\gamma v}{D_\gamma v} +
\pscal{\delta O_\gamma v}{D_\gamma \delta O_\gamma v}\\
 &\leqs \biggpar{-c_3 \sqrt{\lambdac - \abs{\lambda}} +
\frac{c_4\gamma}{\sqrt{\lambdac -\abs{\lambda}}} + 
\frac{c_5\gamma^2}{\lambdac -\abs{\lambda}}} \norm{v}^2
 \label{eq:plg_404}
\end{align} 
for constants $c_3, c_4, c_5 > 0$. This shows that for $\gamma/(\lambdac
-\abs{\lambda})$ sufficiently small, $\pscal{v}{H_\gamma v}<0$ and thus
$x^\star$ is a saddle of index at least $1$ of the constrained system. However,
the index cannot be larger than for the unconstrained system, so that if must
equal $1$. 
\end{proof}

\begin{proof}[Proof of Theorem~\ref{thm:persistence}]
If we denote by $\pm\hat\lambda(\gamma)$ the upper and lower boundaries of $D$,
then a sufficient condition for $\Sigma_\gamma$ to change sign is 
\begin{equation}
 \label{eq:plg_42}
 \Sigma_0(\hat\lambda(\gamma)) > \sqrt{\gamma}
 \qquad\text{and}\qquad
 \Sigma_0(-\hat\lambda(\gamma)) < -\sqrt{\gamma}\;. 
\end{equation} 
Without limiting the generality, we assume $\lambda_0\geqs0$. Then the first of
the two conditions is the more stringent one. For the family $B_k$, using the
fact that $\alphamin(\lambda) = -\frac{1}{\sqrt3} -
\Order{\sqrt{\lambdac-\lambda}\,}$ near $\lambdac$ we obtain 
\begin{equation}
 \label{eq:plg_43}
 \Sigma_0(\lambda) = \frac{1}{\sqrt3} \biggpar{\frac12 - \frac{3k}{N}} 
 - c\biggpar{\frac12+\frac{k}{N}}\sqrt{\lambdac-\lambda} +
\Order{\lambdac-\lambda}
\end{equation} 
for some constant $c>0$. Since we also have
$\hat\lambda(\gamma)\geqs\lambdac(1-\frac92\gamma)=\lambdac-\sqrt3\gamma$,
inserting this in~\eqref{eq:plg_42} yields the result. The case of the
families $C_k$ is similar, noting that the bound on $\gamma$ in
Lemma~\ref{lem:horseshoe2} ensuring that they remain $1$-saddles is fulfilled
under the condition~\eqref{eq:wpc01}. 

In the case of the family $B_0$, one can obtain sharper bounds by first noting
that $x^\star$ has exactly half of its components in $\cV_-$ and the other half
in $\cV_+$. Using the bounds given in Lemma~\ref{lem:horseshoe}, we see
that~\eqref{eq:plg_41} can be strengthened to 
\begin{equation}
 \label{eq:plg_44}
 \Sigma_0(\lambda) - \frac12\sqrt{\gamma} \leqs \Sigma_\gamma(\lambda) 
 \leqs \Sigma_0(\lambda) + \frac12\sqrt{\gamma}\;.
\end{equation} 
Furthermore, we have 
\begin{equation}
 \label{eq:plg_45}
 \Sigma_0(\lambda) = \frac12 \alphamin(\lambda) + \frac12\alphamax(\lambda) =
-\frac12\alphac(\lambda)\;.
\end{equation} 
A sufficient condition for the stationary point to exist is thus 
\begin{equation}
 \label{eq:plg_46}
 -\alphac\bigpar{\lambdac(1-\tfrac92\gamma)} > \sqrt{\gamma}\;.
\end{equation} 
By definition of $\alphac$, this is equivalent to $\lambdac(1-\frac92\gamma) >
\sqrt{\gamma}(\gamma-1)$. Taking the square yields the condition $27\gamma^3 -
135\gamma^2+54\gamma-4 < 0$, which holds for $\gamma < \frac73 - \sqrt5$. 
\end{proof}


\section{Proofs: Metastable hierarchy}
\label{sec_proof_hierarchy}


\subsection{Hierarchy of the $B_k$}
\label{ssec_proof_Bk}

\begin{proof}[Proof of Theorem~\ref{thm:hierarchy_Bk}]
When $\gamma=0$, the value of the potential is constant on each family $B_k$
and $C_k$. Using the expressions~\eqref{eq:alpha_j} of the $\alpha_j$, one
obtains for these values 
\begin{equation}
 \label{eq:pBk01}
 V_0(B_k) = - \frac{aN(N-a)}{4(N^2-3aN+3a^2)}\;, 
 \qquad
 V_0(C_{k+1}) = - \frac{aN^2-(a^2+8a-8)N+9a(a-1)}{4(N^2-3aN+3a^2-3a+3)}\;,
\end{equation} 
where $a=M-k$ in both cases. Taking differences and simplifying yields  
\begin{align}
\label{eq:pBk02} 
V_0(C_{k+1}) - V_0(B_k) 
&= \frac{(a-1)(2N-3a)^3}{4(N^2-3aN+3a^2)(N^2-3aN+3a^2-3a+3)} =: h_1(a)\;, \\
V_0(C_k) - V_0(B_k) 
&= \frac{(N-a-1)(3a-N)^3}{4(N^2-3aN+3a^2)(N^2-3(a+1)N+3a^2+3a+3)} =: h_2(a)\;.
\nonumber
\end{align}
Computing derivatives and proceeding in a similar way as in the proof of
Theorem~\ref{thm:transition_graph}, one obtains that $a\mapsto h_1(a)$ is
decreasing, while $a\mapsto h_2(a)$ is increasing. Furthermore, it is immediate
to check that $h_1(N/2) = h_2(N/2)$. We thus obtain the inequalities 
\begin{equation}
 \label{eq:pBk03}
 \dots < V_0(C_2) - V_0(B_2) < V_0(C_1) - V_0(B_1) 
 < V_0(C_1) - V_0(B_0) < V_0(C_2) - V_0(B_1) < \dots
\end{equation}
(cf.~\figref{fig:BC}).
To prove~\eqref{eq:Bk01}, we have to check that
relation~\eqref{def:metastable_hierarchy} holds for each $B_k$. Indeed, on the
one hand we have 
\begin{equation}
 \label{eq:pBk04}
 H \biggpar{B_k, \bigcup_{i=0}^{k-1}B_i} = V_0(C_k) - V_0(B_k)
\end{equation} 
for $k\geqs 2$, while on the other hand 
\begin{align}
\nonumber
H\biggpar{B_\ell, \bigcup_{i=0}^k B_i\setminus B_\ell}
&\geqs V_0(C_\ell) - V_0(B_\ell)\;, 
&& \ell\in\intint{2}{k-1}\;, \\
H\biggpar{B_0, \bigcup_{i=1}^{k}B_i} &= V_0(C_1) - V_0(B_0)\;.
\label{eq:pBk05} 
\end{align}
Thus the result follows from~\eqref{eq:pBk03}. 

When $\gamma>0$ is sufficiently small, the same partition still forms a
metastable hierarchy, because the potential heights of the critical points
depend continuously on $\gamma$. 
\end{proof}


\subsection{Hierarchy on $B_0$}
\label{ssec_proof_B0}

\begin{proof}[Proof of Proposition~\ref{prop:transitions_B0}]
The fact that allowed transitions between elements in $B_0$ correspond to
exchanging a particle and a hole follow directly from the connection
rules~\eqref{eq:gamma0_rules}. Indeed, any element in $B_1$, with triple
$(0,M-1,M+1)$, is connected to $M+1$ elements of $B_0$, which differ by a
particle/hole transposition (see also~\figref{fig:transition_8}). 
Any such transition affects at most $4$ interfaces. Since the number of
interface is always even, we obtain~\eqref{eq:B0_03}. 

In order to compute communication heights, we have to determine the heights of
$1$-saddles $z^\star$ in $C_1$. Recall that each of these saddles has $1$
coordinate equal to $\alpha_0'$, $M-1$ coordinates equal to $\alpha_1'$ and $M$
coordinates equal to $\alpha_2'$, where 
\begin{equation}
 \label{eq:pB0_01}
 (\alpha_0',\alpha_1',\alpha_2') = \pm\omega'(1,1-M,M-2)\;, 
 \qquad
 \omega' = (M^2-3M+3)^{-1/2}\;.
\end{equation} 
Plugging this into~\eqref{eq:B0_02} yields 
\begin{align}
\nonumber
 V_\gamma\bigpar{z^\star(\gamma)} 
 ={}& V_0\bigpar{z^\star(0)} \\
\nonumber
 &{}+ \frac{\gamma}{4}(\omega')^2 
 \bigbrak{M^2 I_{\alpha'_0/\alpha'_1}(z^\star) 
 + (M-3)^2 I_{\alpha'_0/\alpha'_2}(z^\star) + 
 (2M-3)^2 I_{\alpha'_1/\alpha'_2}(z^\star)} \\
 &{}+ \Order{\gamma^2}\;,
 \label{eq:pB0_02}
\end{align} 
where, similarly to~\eqref{eq:B0_01a}, $I_{\alpha'_i/\alpha'_j}(z^\star)$
denotes the number of interfaces of type $\alpha'_i/\alpha'_j$ of $z^\star$. The
first-order correction to the height of the saddle thus only depends on the
triple 
\begin{equation}
 \label{eq:pB0_03}
 I(z^\star) = \bigbrak{I_{\alpha'_0/\alpha'_1}(z^\star),
I_{\alpha'_0/\alpha'_2}(z^\star), I_{\alpha'_1/\alpha'_2}(z^\star)}\;,
\end{equation} 
where we use square brackets in order to avoid confusion with the triple
$(0,M-1,M+1)$. Note in particular that $I_{\alpha'_0/\alpha'_1}(z^\star) +
I_{\alpha'_0/\alpha'_2}(z^\star) = 2$, since only $1$ component of $z^\star$ is
equal to $\alpha'_0$. The following lemma allows to compare all these saddle
heights. 

\begin{lemma}
\label{lem:saddle_heights}
For all even $p\in\intint{2}{N-2}$, the first-order terms $V^{(1)}$
in~\eqref{eq:pB0_02} satisfy 
\begin{equation}
 \label{eq:pB0_04}
 V^{(1)}\bigpar{[2,0,p]} > V^{(1)}\bigpar{[0,2,p]} > V^{(1)}\bigpar{[1,1,p-1]} >
V^{(1)}\bigpar{[2,0,p-2]}\;,
\end{equation} 
where $[a,b,c]$ stands for any saddle $z^\star$ such that $I(z^\star) =
[a,b,c]$. 
\end{lemma}
\begin{proof}
This follows from a straightforward computation, using~\eqref{eq:pB0_02} and the
fact that $(2M-3)(M-3) > 0$. 
\end{proof}

It remains to apply these expressions to the different transitions in
Table~\ref{table:transitions_B0}. Consider for instance the transition shown in
\figref{fig:transition_8}, which is of type \rom{2}.b. The two saddles
encountered during the transition are of type $[1,1,p-1]$ and $[0,2,p]$, where
$p=2$ is the number of interfaces of the start configuration $x^\star$.
Lemma~\ref{lem:saddle_heights} shows that the second saddle is the highest. 
Combining this with the expression~\eqref{eq:B0_02} of the height of $x^\star$
yields 
\begin{equation}
 \label{eq:pB0_05}
 H^{(1)} = \frac{2(M-3)^2+p(2M-3)^2}{4(M^2-3M+3)} - p 
 = \frac{2(M-3)^2 - 3p}{4(M^2-3M+3)}\;,
\end{equation} 
which is precisely the expression given in the second line of
Table~\ref{table:transitions_B0}. 

\begin{table}[tb]
\begin{center}
\begin{tabular}{|c||c|c|c|c|c|c|}
\hline
\vrule height 11pt depth 4pt width 0pt
    & \rom{1} & \rom{2}.a/b & \rom{3} & \rom{4}.a/b & \rom{5}.a/b & \rom{6} \\
\hline
\hline
\vrule height 11pt depth 4pt width 0pt
$\eta$ & $(0,0)$ & $(0,1)/(1,0)$ & $(1,1)$ & $(0,2)/(2,0)$ & $(1,2)/(2,1)$ &
$(2,2)$ \\
\hline
\vrule height 11pt depth 5pt width 0pt
$I(z^\star)$ & $[0,2,p+2]$ & $[0,2,p]$ & $[1,1,p-1]$ & $[0,2,p-2]$ & $[0,2,p-2]$
& $[0,2,p-2]$ \\
\hline
\end{tabular}
\end{center}
\caption[]{Interface numbers of the highest saddle encountered along a minimal
path between two configurations in $B_0$, if the two exchanged sites $i$ and
$j$ are not nearest neighbours. The labels in the first row are the same as in
Table~\ref{table:transitions_B0}, and $\eta$ denotes the number of interfaces of
$i$ and $j$.}
\label{table:pot_notneighbour} 
\end{table}

\begin{table}[tb]
\begin{center}
\begin{tabular}{|c||c|c|c|}
\hline
\vrule height 11pt depth 4pt width 0pt
    & \rom{2}.c & \rom{4}.c/d & \rom{5}.c \\
\hline
\hline
\vrule height 11pt depth 4pt width 0pt
$\eta$ & $(1,1)$ & $(1,2)/(2,1)$ & $(2,2)$ \\
\hline
\vrule height 11pt depth 5pt width 0pt
$I(z^\star)$ & $[0,2,p]$ & $[0,2,p-2]$ & $[0,2,p-2]$ \\
\hline
\end{tabular}
\end{center}
\caption[]{Interface numbers of the highest saddle encountered along a minimal
path between two configurations in $B_0$, if the two exchanged sites $i$ and
$j$ are nearest neighbours.}
\label{table:pot_neighbour} 
\end{table}

The other cases are treated in a similar way. One just has to take care of the
fact that the transition rules~\eqref{eq:gamma0_rules} allow for two possible
paths, depending on whether $(\alpha_1,\alpha_2)=(1,-1) + \Order{\gamma}$ or
$(-1,1) + \Order{\gamma}$. It is thus necessary to determine the minimum of the
communication heights associated with these two paths.  
Table~\ref{table:pot_notneighbour} shows the associated saddles in cases where
the exchanged sites are not nearest neighbours. Table~\ref{table:pot_neighbour}
shows the same when the exchanged sites are nearest neighbours. These saddle
interface numbers are indeed those shown in Table~\ref{table:transitions_B0}.  
\end{proof}

\begin{proof}[Proof of Theorem~\ref{thm:meta_B0}]
In a similar way as in the proof of Theorem~\ref{thm:hierarchy_Bk}, we prove
that the relation~\eqref{eq:def_meta_hierarchy} holds when the $A_p$ and $A'_p$
are ordered according to~\eqref{eq:B0_hierarchy}. Since all communication
heights are the same when $\gamma=0$, it will be sufficient to compare the
first-order coefficients $H^{(1)}$. 

We start by showing that $A_2 \prec A'_4 \prec \dots \prec A'_{M'}$.
For any $p\in\intint{4}{M'}$, we note that 
\begin{equation}
 \label{eq:pB0_10}
 H^{(1)}(A'_p, A_2\cup A'_4\cup \dots \cup A'_{p-2}) 
 = H^{(1)}(A'_p, A_p) 
 = \frac{-2M^2+6M-3p}{4(M^2-3M+3)}\;.
\end{equation} 
Indeed, the highest saddle encountered along a minimal path from $A'_p$ to
$A_2\cup A'_4\cup \dots A'_{p-2}$ occurs during the type-\rom{3} transition from
$A'_p$ to $A_p$, and has interface number $[1,1,p-1]$. Since~\eqref{eq:pB0_10}
is a decreasing function of $p$, the condition~\eqref{eq:def_meta_hierarchy} is
indeed satisfied. 

Next we observe that 
\begin{equation}
 \label{eq:pB0_11}
 H^{(1)}(A_4, A_2\cup A'_4\cup \dots \cup A'_{M'}) 
 = H^{(1)}(A_4, A_2) 
 = \frac{-6M^2+12M-12}{4(M^2-3M+3)}\;.
\end{equation} 
Indeed, here the minimal path goes directly from $A_4$ to $A_2$, via a saddle
of type $[0,2,2]$. The expression~\eqref{eq:pB0_11} is indeed smaller
than~\eqref{eq:pB0_10} for $p=M'$, the numerator of the difference being
bounded by $4M^2-9M+12$ which is always positive. 

Finally, we see that we have 
\begin{equation}
 \label{eq:pB0_12}
 H^{(1)}(A_p, A_2\cup \dots\cup A'_{M'}\cup A_4\cup\dots\cup A_{p-2}) 
 = H^{(1)}(A_p, A_{p-2}) 
 = \frac{-6M^2+12M-3p}{4(M^2-3M+3)}\;,
 \end{equation} 
the minimal path reaching communication height on a saddle of type $[0,2,p-2]$.
Since~\eqref{eq:pB0_12} is again a decreasing function of $p$, the claim
follows.
\end{proof}


\section{Proofs: Spectral gap}
\label{sec_proof_gap}

In order to prove Theorem~\ref{thm:spectral_gap}, we have to take into account
the symmetries of the potential $V_\gamma$. This will allow us to apply the
theory in~\cite{BD15} on metastable processes that are invariant under a group
of symmetries, which relies on Frobenius' representation theory of finite groups
(see for
instance~\cite{Serre_groups}). 

The potential $V_\gamma$ is invariant under the three transformations 
\begin{align}
\nonumber 
r: (x_1,\dots,x_N) &\mapsto (x_2,\dots,x_N,x_1)\;, \\
\label{eq:psg01} 
s: (x_1,\dots,x_N) &\mapsto (x_N,\dots,x_1)\;, \\
c: (x_1,\dots,x_N) &\mapsto (-x_1,\dots,-x_N)\;.
\nonumber 
\end{align}
It is thus invariant under the group $G$ generated by these three
transformations. This group can be written $G=\mathfrak{D}_N\times\Z_2$, where
$\mathfrak{D}_N$ is the dihedral group of symmetries of a regular $N$-gon
generated by $r$ and $s$, while $\Z_2=\set{\id,c}$ is the group generated by
$c$, which commutes with $r$ and $s$. The group $G$ has order $4N$, and its
elements can be written $r^i s^j c^k$ with $i\in\intint{0}{N-1}$ and $j,
k\in\set{0,1}$. It admits exactly $8$ one-dimensional irreducible
representations given by 
\begin{equation}
 \label{eq:psg02}
 \pi_{\rho\sigma\tau}(r^i s^j c^k) = \rho^i\sigma^j\tau^k\;, 
 \qquad \rho, \sigma, \tau = \pm1\;,
\end{equation} 
and $N-2$ irreducible representation of dimension $2$, whose characters are 
\begin{equation}
 \label{eq:psg03}
 \chi_{\ell,\pm}(r^i s^j c^k) 
 = \Tr \pi_{\ell,\pm}(r^i s^j c^k) 
 = 2 \cos\biggpar{\frac{2i\ell\pi}{N}}
 \delta_{j0}(\pm1)^k\;, 
 \qquad \ell \in\intint{1}{\tfrac N2-1}\;.
\end{equation} 
The basic idea of the approach given in~\cite{BD15} is that each of these
$N+6$ irreducible representations provides an invariant subspace of the
generator $L$ of the Markovian jump process on $\cS_0$ that approximates the
dynamics of the diffusion. Thus the restriction of $L$ to each of these
subspaces yields a part of the spectrum of $L$. The eigenvalues of $L$ can then
be shown to be close to the exponentially small eigenvalues of
$\cL$~\cite{Dutercq_PhD,Dutercq_in_preparation}. We thus have to determine,
for each irreducible representation, the smallest eigenvalue of $-L$. We do
this in two main steps: first we show that the Arrhenius exponent of each
smallest nonzero eigenvalue is given by the potential difference between
certain stationary points in $C_1$ and in $A_2$, and then we compute the
smallest prefactor of these eigenvalues. 


\subsection{Arrhenius exponent}
\label{ssec_gap_arrhenius}

Each $g\in G$ induces a permutation $\pi_g$ on the set of local minima $\cS_0$,
leaving invariant each group orbit $O_a = \setsuch{ga}{g\in G}$. Since
$V_\gamma$ is $G$-invariant, the generator $L$ commutes with all these
permutations. Thus there exist subspaces which are jointly invariant under $L$
and all the $\pi_g$. Each of the irreducible representations of $G$ provides one
of these subspaces.  

Let $\pi$ be one of the irreducible representations of $G$, and let
$d\in\set{1,2}$ be its dimension. Then~\cite[Lemma~3.6]{BD15} shows that the
associated invariant subspace, when restricted to $O_a$, has dimension
$d\alpha^\pi_a$, where 
\begin{equation}
 \label{eq:gap_arr01}
 \alpha^\pi_a = \frac{1}{\abs{G_a}} \sum_{h\in G_a} \chi(h) \in\intint{0}{d}\;. 
\end{equation} 
Here $\chi = \Tr\pi$ denotes the character of $\pi$, and $G_a=\setsuch{g\in
G}{ga=a}$ the stabiliser of $a$. We call \emph{active} with respect to the
irreducible representation $\pi$ the orbits $O_a$ such that $\alpha^\pi_a > 0$.
Only active orbits will occur in the restriction of $L$ to the invariant
subspace associated with $\pi$; they are represented by a block of size
$d\alpha^\pi_a \times d\alpha^\pi_a$. 

We select three representatives $x^\star\in A_2$, $y^\star\in B_1$ and
$z^\star\in C_1$ such that $x^\star$ is connected to $y^\star$ via $z^\star$ in
the transition graph $\cG$. A possible choice is 
\begin{align}
\nonumber
x^\star &= (\alpha_1,\dots,\alpha_1,\alpha_1,\alpha_2\dots,\alpha_2)\;, \\
\label{eq:gap_arr02}
z^\star &= (\alpha'_1,\dots,\alpha'_1,\alpha'_0,\alpha'_2\dots,\alpha'_2)\;, \\
y^\star &=
(\alpha''_1,\dots,\alpha''_1,\alpha''_2,\alpha''_2\dots,\alpha''_2)\;,
\nonumber
\end{align}
where $\alpha_1=1+\Order{\gamma}$, $\alpha_2=-1+\Order{\gamma}$ and $\alpha'_2$
are each repeated $M$ times, $\alpha'_1$ and $\alpha''_1$ are repeated $M-1$
times and $\alpha''_2$ is repeated $M+1$ times. The orbit $O_x$ of $x^\star$ is
precisely $A_2$, and it has $N$ elements. The orbits $O_y$ of $y^\star$ and
$O_z$ of $z^\star$ have respectively $2N$ and $4N$ elements (they are proper
subsets of $B_1$ and $C_1$). The associated stabilisers are given by 
\begin{align}
\nonumber
G_x &= \set{\id, r^Ms, r^Mc, sc}\;, \\
\label{eq:gap_arr03}
G_y &= \set{\id, r^{M-1}s}\;, \\
G_z &= \set{\id}\;,
\nonumber
\end{align}
where $\id$ denotes the identity of $G$ and $M=\frac N2$. Note in particular
that 
\begin{equation}
 \label{eq:gap_arr04}
 \frac{\abs{G_x}}{\abs{G_x\cap G_y}} = 4\;, 
 \qquad 
 \frac{\abs{G_y}}{\abs{G_x\cap G_y}} = 2\;.
\end{equation} 
This means that each element in $G_x$ is connected to $4$ elements in $G_y$
(via $4$ saddles in $G_z$), and that each element in $G_y$ is connected to $2$
elements in $G_x$ (cf.~\cite[(2.25)]{BD15}). 

The possible Arrhenius exponents of eigenvalues of $L$ are directly linked to
which orbits are active for the different irreducible representations. We start
with irreducible representations of dimension $1$, cf.~\eqref{eq:psg02}. 

\begin{prop}
\label{prop:active_1}  
Let $\pi$ be a $1$-dimensional irreducible representation of $G$. Then 
\begin{itemiz}
\item 	if $M$ is even, then $O_x = A_2$ is active if and only if 
$\pi(s)=\pi(c)=1$;
\item 	if $M$ is odd, then $O_x = A_2$ is active if and only if 
$\pi(r)=\pi(s)=\pi(c)$;
\item 	if $M$ is even, then $O_y$ is active if and only if 
$\pi(r)=\pi(s)$;
\item 	if $M$ is odd, then $O_y$ is active if and only if 
$\pi(s)=1$.
\end{itemiz}
\end{prop}
\begin{proof}
The orbit $O_x$ is active if and only if $\pi(g)=1$ for all $g\in G_x$. Since
$\pi(\id)=1$, $\pi(r^Ms)=\pi(r)^M\pi(s)$, $\pi(r^Mc)=\pi(r)^M\pi(c)$ and
$\pi(sc)=\pi(s)\pi(c)$, this holds if and only if $\pi(r)^M=\pi(s)=\pi(c)$. 
Similarly, the orbit $O_y$ is active if and only if $\pi(r)^{M-1}=\pi(s)$.
\end{proof}

The corresponding result for the $2$-dimensional irreducible
representations given in~\eqref{eq:psg03} reads as follows.

\begin{prop}
\label{prop:active_2}  
Let $\pi_{\ell,\pm}$ be a $2$-dimensional irreducible representation. Then
\begin{itemiz}
\item 	$O_x=A_2$ is active for $\pi_{\ell,+}$ if and only if $\ell$ is even; 
\item 	$O_x=A_2$ is active for $\pi_{\ell,-}$ if and only if $\ell$ is odd;
\item 	$O_y$ is active for all representations $\pi_{\ell,\pm}$.
\end{itemiz}
\end{prop}
\begin{proof}
By~\eqref{eq:psg03} we have $\chi_{\ell,\pm}(\id)=2$, $\chi_{\ell,\pm}(r^Ms)=0$,
$\chi_{\ell,\pm}(r^Mc)=\pm2\cos(\ell\pi)$ and $\chi_{\ell,\pm}(sc)=0$. 
Thus the sum~\eqref{eq:gap_arr01} for $O_x$ is different from $0$ for
$\chi_{\ell,+}$ if and only if $\ell$ is even, and for $\chi_{\ell,-}$ if and
only if $\ell$ is odd. Since $\chi_{\ell,\pm}(r^{M-1}s)=0$, the sum for $O_y$
is always equal to $1$.  
\end{proof}

\begin{cor}
The maximal Arrhenius exponent of all nonzero eigenvalues of the generator is
given by $V_\gamma(z^\star) - V_\gamma(x^\star)$.
\end{cor}
\begin{proof}
\cite[Thm.~3.5]{BD15} provides an algorithm determining the Arrhenius exponents
for each irreducible representation $\pi$ of dimension $1$. They are obtained by
replacing all inactive orbits by a cemetery state, which is at the bottom of
the metastable hierarchy, and ordering all other orbits according to the usual
hierarchy. If $O_x=A_2$ is active and $O_y$ is inactive for $\pi$, then the
largest communication height determining an Arrhenius exponent will be given by
$V_\gamma(z^\star) - V_\gamma(x^\star)$. If $M$ is even, the representation
$\pi_{-++}$ has the required property, while if $M$ is odd, this r\^ole is
played by $\pi_{---}$. 

In the case of $2$-dimensional representations, \cite[Thm.~3.9]{BD15} shows
that all communication heights between active orbits yield Arrhenius exponents.
The largest such exponent is obtained if $O_x$ and $O_y$ are both active, and
Proposition~\ref{prop:active_2} shows that there are representations for which
this is the case. 
\end{proof}


\subsection{Eyring--Kramers prefactor}
\label{ssec_gap_EK}

It remains to find the smallest prefactor associated with a transition of
communication height $V_\gamma(z^\star) - V_\gamma(x^\star)$. In the case of
one-dimensional representations, \cite[Prop.~3.4]{BD15} shows that the usual
Eyring--Kramers law given in Theorem~\ref{thm:BGK} has to be corrected by a
factor $\abs{G_x}/\abs{G_x\cap G_y} = 4$ (cf.~\eqref{eq:gap_arr04}). 

In the case of two-dimensional irreducible representations $\pi_{\ell,\pm}$, the
relevant matrix elements are given in~\cite[Prop.~3.7]{BD15}. Alternatively, one
can compute these elements \lq\lq by hand\rq\rq\ in the following way. We start
by ordering the elements of the two orbits $O_x$ and $O_y$ according to 
\begin{align}
\nonumber
O_x &= \set{x^\star,rx^\star,\dots,r^{N-1}x^\star}\;, \\
O_y &= \set{y^\star,ry^\star,\dots,r^{N-1}y^\star,
r^Mcy^\star,r^{M+1}cy^\star,\dots,r^{M-1}cy^\star}\;. 
\label{eq:gap_EK01} 
\end{align}
The restriction of $L$ to $O_x\cup O_y$ consists in the four blocks 
\begin{align}
\nonumber
L_{xx} &= -4q_x\one_N\;, 
\qquad L_{yy} = -2q_y\one_{2N}\;,\\
L_{xy} &= q_x 
\begin{pmatrix}
1 & 1 &  & (0) & 1 & 1 &  & (0)\\
 & \ddots & \ddots &  &  & \ddots & \ddots & \\
 & (0) & \ddots & 1 &  & (0) & \ddots & 1\\
1 & & & 1 & 1 & & & 1
\end{pmatrix}\;, 
&L_{yx} &= q_y
\begin{pmatrix}
1 &  &  & 1\\
1 & \ddots & (0)  & \\
 & \ddots & \ddots & \\
(0) & & 1 & 1\\
1 &  &  & 1\\
1 & \ddots & (0)  & \\
 & \ddots & \ddots & \\
(0) & & 1 & 1
\end{pmatrix}\;,
\label{eq:gap_EK02} 
\end{align}
where $\one_n$ denotes the $n\times n$ identity matrix, 
$(0)$ stands for repeated zero entries, and 
\begin{equation}
 \label{eq:gap_EK03}
 q_x = 
 \frac{\abs{\lambda_-(z^\star)}}{2\pi} 
 \sqrt{\frac{\det\nabla^2V_\gamma(x^\star)}
 {\det\nabla^2V_\gamma(z^\star)}}
\e^{-\brak{V_\gamma(z^\star)-V_\gamma(x^\star)}/\eps} 
\bigbrak{1 + \Order{\eps^{1/2}\abs{\log\eps}^{3/2}}}\;,
\end{equation} 
while $q_y$ is a positive constant of order
$\e^{-\brak{V_\gamma(z^\star)-V_\gamma(y^\star)}/\eps}$. 
Equation~(3.14) in~\cite{BD15} provides a set of vectors spanning the invariant
subspaces associated with a given irreducible representation. Among these, we
have to choose two linearly independent vectors for each orbit. A possible
choice is 
\begin{align}
\nonumber
u^x &= (2,\chi(r),\ldots ,\chi(r^{N-1}),0,\dots,0)\;,\\
\nonumber
u^{rx} &= (\chi(r^{N-1}),2,\chi(r),\ldots ,\chi(r^{N-2}),0,\dots,0)\;,\\
\nonumber
u^y &= (0,\dots,0, 2,\chi(r), \ldots, \chi(r^{N-1}), 2, \chi(r), \ldots,
\chi(r^{N-1}))\;,\\
u^{ry} &=
(0,\dots,0,\chi(r^{N-1}),2,\chi(r),\ldots, \chi(r^{N-2}), \chi(r^{N-1}), 2,
\chi(r),
\ldots, \chi(r^{N-2}))\;,
\label{eq:gap_EK04} 
\end{align}
where $\chi = \chi_{\ell,\pm}$ is given by~\eqref{eq:psg03}. In this basis, $L$
takes the block form
\begin{align}
\nonumber
L_{xx}^\pi &= -q_x
\begin{pmatrix}
4 & 0 \\
0 & 4
\end{pmatrix}\;,
&L_{yy}^\pi &= -q_y
\begin{pmatrix}
2 & 0 \\
0 & 2
\end{pmatrix}\;,\\
L_{xy}^\pi &= q_x 
\begin{pmatrix}
2(\chi(r)+1) & 2 \\
-2 & 2
\end{pmatrix}\;,
&L_{yx}^\pi &= q_y 
\begin{pmatrix}
1 & -1 \\
1 & (\chi(r)+1)
\end{pmatrix}\;.
\label{eq:gap_EK05} 
\end{align}
We can now apply~\cite[Thm.~3.9]{BD15}, which states that the eigenvalues are
equal to those of 
\begin{equation}
L_{xx}^\pi-L_{xy}^\pi\bigpar{L_{yy}^\pi}^{-1}L_{yx}^\pi =
-4\sin^2\biggpar{\frac{\ell\pi}{N}}q_x\one_2\;.
\label{eq:gap_EK06} 
\end{equation} 
The result~\eqref{eq:sg01} follows, since the minimal value of the eigenvalues
is reached for $\ell=1$, both orbits are active for the representation
$\pi_{1,-}$, and this value is smaller than for all one-dimensional
representations.

\begin{remark}
It is of course possible to obtain the same result directly from the
expressions~(3.15) and~(3.17) given in~\cite{BD15} for the inner products 
$\pscal{u}{Lv}$, where $u$ and $v$ are basis vectors among~\eqref{eq:gap_EK04},
even though these vectors are not orthogonal. It suffices to use the fact that
the matrix elements of $L$ can be obtained by computing 
\begin{equation}
 \label{eq:gap_EK07}
 \begin{pmatrix}
 \vrule height 8pt depth 14pt width 0pt
 1 & \dfrac{\pscal{u}{v}}{\pscal{u}{u}} \\
 \vrule height 14pt depth 8pt width 0pt
 \dfrac{\pscal{u}{v}}{\pscal{v}{v}} & 1
 \end{pmatrix}^{-1}
 \begin{pmatrix}
 \vrule height 8pt depth 14pt width 0pt
 \dfrac{\pscal{u}{Lu}}{\pscal{u}{u}}
 & \dfrac{\pscal{u}{Lv}}{\pscal{u}{u}} \\
 \vrule height 14pt depth 8pt width 0pt
 \dfrac{\pscal{v}{Lu}}{\pscal{v}{v}} 
 & \dfrac{\pscal{v}{Lv}}{\pscal{v}{v}}
 \end{pmatrix}
\end{equation} 
where for instance $\pscal{u^x}{u^{rx}} =
\cos(\ell\pi/M)\pscal{u^x}{u^x}$.~$\lozenge$
\end{remark}

The last element of the proof of Theorem~\ref{thm:spectral_gap} is the
following result on the Hessian matrices of $V_0$. 

\begin{prop}
\label{prop:hessian_V0}
The Hessian matrices of $V_0$ at $x^\star$ and $z^\star$ satisfy 
\begin{align}
\nonumber
 \det\nabla^2 V_0(x^\star) &= 2^{N-1}\;, \\
\nonumber
 \det\nabla^2 V_0(z^\star) 
 &= -\frac{M^{M-2}(M-3)^M(2M-3)^{2M-2}}{(M^2-3M+3)^{2M-2}} 
 = -2^{N-2}\bigbrak{1+\Order{N^{-1}}}\;,\\
 \lambda_-(z^\star) 
 &= -\frac{(M-3)(2M-3)}{2(M^2-3M+3)} 
 = -1 + \Order{N^{-1}}\;.
 \label{eq:p-hess01}
\end{align} 
\end{prop}
\begin{proof}
We have already obtained invariant subspaces of the Hessian matrices in the
proof of Proposition~\ref{prop_saddles}. However, since we used a non-isometric
parametrisation of $S$, we cannot use expressions such as~\eqref{eq:pl0_05}
directly to determine the eigenvalues. 

In the case of $x^\star=(1,\dots,1,-1,\dots,-1)$, it is sufficient to note that
for any vector $u$ of unit length in $S$, one has 
\begin{equation}
 \label{eq:p-hess02}
 \frac{\6^2}{\6t^2} V_0(x^\star + tu) \biggr\vert_{t=0}
 = \sum_{i=1}^M u_i^2 U''(1) + \sum_{i=M+1}^N u_i^2 U''(-1)
 = 2\;,
\end{equation} 
showing that in fact $\nabla^2 V_0(x^\star)=2\one_{N-1}$, which has determinant
$2^{N-1}$.  

In the case of the saddle $z^\star$ given by~\eqref{eq:gap_arr02}, the
expressions~\eqref{eq:alpha_j} for the $\alpha'_j$ yield 
\begin{alignat}{3}
\nonumber
 U''(\alpha'_0) &= 3(\alpha'_0)^2-1 
 &&= -\frac{M(M-3)}{M^2-3M+3}
 &&= -1 + \Order{M^{-2}}\;, \\
\nonumber
 U''(\alpha'_1) &= 3(\alpha'_1)^2-1 
 &&= \frac{M(2M-3)}{M^2-3M+3}
 &&= 2 + 3M^{-1} + \Order{M^{-2}}\;, \\
 U''(\alpha'_2) &= 3(\alpha'_2)^2-1 
 &&= \frac{(M-3)(2M-3)}{M^2-3M+3}
 &&= 2 - 3M^{-1} + \Order{M^{-2}}\;. 
 \label{eq:p-hess03}
\end{alignat} 
We know from the proof of Proposition~\ref{prop_saddles} that the
($M-2$)-dimensional subspace of $S$ given by 
$S_1=\set{x_1+\dots+x_{M-1}=0, x_M=\dots=x_N=0}$ is invariant by the Hessian.
Proceeding as in~\eqref{eq:p-hess02} with a unit vector $u\in S_1$ shows that
$U''(\alpha'_1)$ is an eigenvalue of $\nabla^2 V_0(z^\star)$ of multiplicity
$M-2$. In an analogous way, the ($M-1$)-dimensional invariant subspace of $S$
given by $S_2=\set{x_1=\dots=x_M=0, x_{M+1}+\dots+x_N=0}$ carries the
eigenvalue $U''(\alpha'_2)$ with a multiplicity $M-1$. This leaves a
two-dimensional invariant subspace, for which we may choose the orthonormal
basis given by the vectors 
\begin{align}
\nonumber
\hat v &= \tfrac{1}{\sqrt{2M}} (1, \dots, 1, 1, -1, \dots, -1)\;, \\
\hat w &= \tfrac{1}{\sqrt{M(M-1)}} (-1, \dots, -1, M-1, 0, \dots, 0)\;.
\label{eq:p-hess04}
\end{align}
The Hessian at $(0,0)$ of the map $(t,s)\mapsto V_0(z^\star+t\hat
v+s\hat w)$ is found to be the matrix
\begin{equation}
 \label{eq:p-hess05}
 \frac{1}{2(M^2-3M+3)}
 \begin{pmatrix}
 4M^2 -15M + 15 & -3\sqrt{2(M-1)}(M-2) \\ 
 -3\sqrt{2(M-1)}(M-2) & -2(M^2-6M+6)
 \end{pmatrix}\;,
\end{equation} 
which has eigenvalues 
\begin{equation}
 \label{eq:p-hess06}
 2 
 \qquad \text{and} \qquad
 - \frac{(M-3)(2M-3)}{2(M^2-3M+3)}
 = -1 + \Order{N^{-1}}\;.
\end{equation}
The result follows via a Taylor expansion of
$\log(-\det\nabla^2V_0(z^\star))$. 
\end{proof}



{\small
\bibliography{BD}
\bibliographystyle{abbrv}               
}


\tableofcontents

\vfill

\bigskip\bigskip\noindent
{\small
Universit\'e d'Orl\'eans, Laboratoire {\sc Mapmo} \\
{\sc CNRS, UMR 7349} \\
F\'ed\'eration Denis Poisson, FR 2964 \\
B\^atiment de Math\'ematiques, B.P. 6759\\
45067~Orl\'eans Cedex 2, France \\
{\it E-mail addresses: }{\tt nils.berglund@univ-orleans.fr}, 
{\tt sebastien.dutercq@univ-orleans.fr}


\end{document}

%% file: fig_unconstained_4.tex
\begin{tikzpicture}
\tikzset{myline/.style={draw=white,double=black,double distance=1pt,very thick}}

\node[white,shape=rectangle, minimum size=1.6cm] (0000) at (0,0) {};
\pos{0}{0};
\isquare{white}{white}{white}{white}{0.4}{3.5pt};


\node[white,shape=rectangle, minimum size=1.6cm] (1000) at (3.5,3) {};
\pos{3.5}{3};
\isquare{black}{white}{white}{white}{0.4}{3.5pt};

\node[white,shape=rectangle, minimum size=1.6cm] (0100) at (3.5,1) {};
\pos{3.5}{1};
\isquare{white}{black}{white}{white}{0.4}{3.5pt};

\node[white,shape=rectangle, minimum size=1.6cm] (0010) at (3.5,-1) {};
\pos{3.5}{-1};
\isquare{white}{white}{black}{white}{0.4}{3.5pt};

\node[white,shape=rectangle, minimum size=1.6cm] (0001) at (3.5,-3) {};
\pos{3.5}{-3};
\isquare{white}{white}{white}{black}{0.4}{3.5pt};

\draw[myline] (0000) -- (1000);
\draw[myline] (0000) -- (0100);
\draw[myline] (0000) -- (0010);
\draw[myline] (0000) -- (0001);


\node[white,shape=rectangle, minimum size=1.6cm] (1100) at (7,2) {};
\pos{7}{2};
\isquare{black}{black}{white}{white}{0.4}{3.5pt};

\node[white,shape=rectangle, minimum size=1.6cm] (0110) at (7,0) {};
\pos{7}{0};
\isquare{white}{black}{black}{white}{0.4}{3.5pt};

\node[white,shape=rectangle, minimum size=1.6cm] (0011) at (7,-2) {};
\pos{7}{-2};
\isquare{white}{white}{black}{black}{0.4}{3.5pt};

\node[white,shape=rectangle, minimum size=1.6cm] (1001) at (7,-4) {};
\pos{7}{-4};
\isquare{black}{white}{white}{black}{0.4}{3.5pt};

\draw[thick] (1000.south east) -- (1001.north
west);
\draw[myline] (1000) -- (1100);
\draw[myline] (0100) -- (1100);
\draw[myline] (0100) -- (0110);
\draw[myline] (0010) -- (0110);
\draw[myline] (0010) -- (0011);
\draw[myline] (0001) -- (0011);
\draw[myline] (0001) -- (1001);


\node[white,shape=rectangle, minimum size=1.6cm] (1101) at (10.5,3) {};
\pos{10.5}{3};
\isquare{black}{black}{white}{black}{0.4}{3.5pt};

\node[white,shape=rectangle, minimum size=1.6cm] (1110) at (10.5,1) {};
\pos{10.5}{1};
\isquare{black}{black}{black}{white}{0.4}{3.5pt};

\node[white,shape=rectangle, minimum size=1.6cm] (0111) at (10.5,-1) {};
\pos{10.5}{-1};
\isquare{white}{black}{black}{black}{0.4}{3.5pt};

\node[white,shape=rectangle, minimum size=1.6cm] (1011) at (10.5,-3) {};
\pos{10.5}{-3};
\isquare{black}{white}{black}{black}{0.4}{3.5pt};

\draw[thick] (1001.north east) -- (1101.south west);
\draw[myline] (1100) -- (1101);
\draw[myline] (1100) -- (1110);
\draw[myline] (0110) -- (1110);
\draw[myline] (0110) -- (0111);
\draw[myline] (0011) -- (0111);
\draw[myline] (0011) -- (1011);
\draw[myline] (1001) -- (1011);


\node[white,shape=rectangle, minimum size=1.6cm] (1111) at (14,0) {};
\pos{14}{0};
\isquare{black}{black}{black}{black}{0.4}{3.5pt};

\draw[myline] (1110) -- (1111);
\draw[myline] (0111) -- (1111);
\draw[myline] (1011) -- (1111);
\draw[myline] (1101) -- (1111);

\end{tikzpicture}

%% file: fig_constrained_4.tex
\begin{tikzpicture}

\def\octasize{5}

\pos{\octasize*cos{\angxy}+0.5}{-\octasize*sin{\angxy}*sin{\angxz}};
\isquare{black}{black}{white}{white}{0.4}{3.5pt};

\pos{-\octasize*sin{\angxy}-0.5}{-\octasize*cos{\angxy}*sin{\angxz}+0.5};
\isquare{white}{black}{black}{white}{0.4}{3.5pt};

\pos{-\octasize*cos{\angxy}-0.5}{\octasize*sin{\angxy}*sin{\angxz}};
\isquare{white}{white}{black}{black}{0.4}{3.5pt};

\pos{0}{\octasize*cos{\angxz}+0.5};
\isquare{black}{white}{black}{white}{0.4}{3.5pt};

\pos{0}{-\octasize*cos{\angxz}-0.5};
\isquare{white}{black}{white}{black}{0.4}{3.5pt};

\draw[thick] (-\octasize*cos{\angxy},\octasize*sin{\angxy}*sin{\angxz}) --
(\octasize*sin{\angxy},\octasize*cos{\angxy}*sin{\angxz}) --
(\octasize*cos{\angxy},-\octasize*sin{\angxy}*sin{\angxz});
\draw[thick,dash pattern=on 2mm off 1mm]
(\octasize*cos{\angxy},-\octasize*sin{\angxy}*sin{\angxz}) --
(-\octasize*sin{\angxy},-\octasize*cos{\angxy}*sin{\angxz}) --
(-\octasize*cos{\angxy},\octasize*sin{\angxy}*sin{\angxz});
\draw[thick] (0,\octasize*cos{\angxz}) --
(\octasize*sin{\angxy},\octasize*cos{\angxy}*sin{\angxz}) --
(0,-\octasize*cos{\angxz});
\draw[thick] (0,\octasize*cos{\angxz}) --
(\octasize*cos{\angxy},-\octasize*sin{\angxy}*sin{\angxz}) --
(0,-\octasize*cos{\angxz});
\draw[thick,dash pattern=on 2mm off 1mm] (0,\octasize*cos{\angxz}) --
(-\octasize*sin{\angxy},-\octasize*cos{\angxy}*sin{\angxz}) --
(0,-\octasize*cos{\angxz});
\draw[thick] (0,\octasize*cos{\angxz}) --
(-\octasize*cos{\angxy},\octasize*sin{\angxy}*sin{\angxz}) --
(0,-\octasize*cos{\angxz});

\pos{\octasize*sin{\angxy}+0.5}{\octasize*cos{\angxy}*sin{\angxz}-0.5};
\isquare{black}{white}{white}{black}{0.4}{3.5pt};


\pos{0.5*\octasize*cos{\angxy}+0.35}
{0.5*(\octasize*cos{\angxz}-\octasize*sin{\angxy}*sin{\angxz})+0.35};
\isquare{black}{blue!50}{blue!50}{white}{0.3}{2.5pt};

\pos{0.5*\octasize*sin{\angxy}+0.25}
{0.5*\octasize*(cos{\angxz}+cos{\angxy}*sin{\angxz})};
\isquare{black}{white}{blue!50}{blue!50}{0.3}{2.5pt};


\pos{0.5*\octasize*(sin{\angxy}-cos{\angxy})}
{0.5*\octasize*(cos{\angxy}+sin{\angxy})*sin{\angxz}-0.35};
\isquare{blue!50}{white}{blue!50}{black}{0.3}{2.5pt};

\pos{-0.5*\octasize*cos{\angxy}-0.35}
{0.5*\octasize*(sin{\angxy}*sin{\angxz}-cos{\angxz})-0.35};
\isquare{white}{blue!50}{blue!50}{black}{0.3}{2.5pt};




\end{tikzpicture}

%% file: fig_transition_8.tex
\begin{tikzpicture}
\tikzset{mynode/.style={draw=white,shape=rectangle, minimum size=3cm}}

\def\octosite{4pt}

\node[mynode] (B0) at (0,0) {\Large $B_0$};
\pos{0}{0};
\iocto{white}{white}{black}{white}{black}{black}{black}{white}{1};

\node[mynode] (C1) at (5,0) {\Large $C_1$};
\pos{5}{0};
\iocto{\ogreen}{\ogreen}{\oviol}{\ogreen}
{\ooran}{\oviol}{\oviol}{\ogreen}{1};

\node[mynode] (B1) at (10,0) {\Large $B_1$};
\pos{10}{0};
\iocto{\oblue}{\oblue}{\ored}{\oblue}{\oblue}{\ored}{\ored}{\oblue}{1};

\draw[thick] (B0) -- (C1) -- (B1); 

\end{tikzpicture}

%% file: fig_disconnectivity.tex
\begin{tikzpicture}[>=stealth',main node/.style={draw,circle,fill=white,minimum
size=3pt,inner sep=0pt},x=1cm,y=1.2cm
]



\draw[teal,thick,dashed] (-0.58,2.95) -- (4.4,2.95);
\draw[teal,thick,dashed] (-0.58,2.05) -- (1.45,2.05);
\draw[teal,thick,dashed] (4.4,1.05) -- (6.09,1.05);


\draw[black,thick] plot[smooth,tension=.6]
  coordinates{(-2.5,4.5) (-0.7,0) (0.6,2) (1.5,1) (2.7,2.9) (4.3,-0.95)
(5.5,1) (6.25,0.5) (7.5,4.5)};
  

\draw[red,ultra thick] (-0.58,-0.05) -- 
node[red,left,xshift=0.1cm,yshift=0.1cm] {$H_2$} (-0.58,2.95); 

\draw[red,ultra thick] (1.45,0.975) -- 
node[red,left,xshift=0.07cm] {$H_3$} (1.45,2.05); 

\draw[red,ultra thick] (4.4,-1) -- (4.4,4.5);

\draw[red,ultra thick] (6.09,0.35) -- 
node[red,right,xshift=-0.05cm] {$H_4$} (6.09,1.05); 


\node[main node,blue,fill=white,semithick] (2) at (-0.58,-0.05) {}; 
\node[blue] [below of=2,yshift=0.7cm] {$x^\star_2$};

\node[main node,blue,fill=white,semithick] (3) at (1.45,0.975) {}; 
\node[blue] [below of=3,yshift=0.7cm] {$x^\star_3$};

\node[main node,blue,fill=white,semithick] (1) at (4.4,-1) {}; 
\node[blue] [below of=1,yshift=0.7cm] {$x^\star_1$};

\node[main node,blue,fill=white,semithick] (4) at (6.09,0.35) {}; 
\node[blue] [below of=4,yshift=0.7cm] {$x^\star_4$};


\node[teal] at (0.7,2.3) {$z^\star_3$};
\node[teal] at (2.6,3.2) {$z^\star_2$};
\node[teal] at (5.6,1.3) {$z^\star_4$};

\end{tikzpicture}

%% file: fig_BC2.tex
\begin{tikzpicture}[>=stealth',main node/.style={circle,minimum
size=0.25cm,fill=blue!20,draw},x=1.2cm,y=1.2cm,
declare function={
pota(\x) = (810/37829)*\x + (436347/37829)*(\x^2) - (172354/37829)*(\x^3);
potb(\x) = -(810/37829)*\x + (436347/37829)*\x^2 - (14959/1991)*\x^3;
potc(\x) = -567829/37829 + (1702677/37829)*\x - (1267140/37829)*\x^2 +
(283608/37829)*\x^3;
potd(\x) = 3963243/37829 - (5093931/37829)*\x + (2131164/37829)*\x^2 -
(282776/37829)*\x^3;
pote(\x) = - 11233356/37829 + (10102668/37829)*\x - (2934369/37829)*\x^2 +
(280061/37829)*\x^3;
potf(\x) = 23972660/37829 - (16301844/37829)*\x + (3666759/37829)*\x^2 -
(270033/37829)*\x^3;
potg(\x) = - 3532815/3439 + (21398331/37829)*\x - (352116/3439)*\x^2 +
(12244/1991)*\x^3;
poth(\x) = - 0.05+39663891/37829 - (17864097/37829)*\x + (2670462/37829)*\x^2 -
(130905/37829)*\x^ 3;
}]

\newcommand*{\nsamples}{15}
\newcommand*{\myvscale}{0.6}

\pgfmathsetmacro{\zca}{\myvscale*4}
\pgfmathsetmacro{\zcb}{\myvscale*6.1}
\pgfmathsetmacro{\zcc}{\myvscale*10.25}
\pgfmathsetmacro{\zba}{\myvscale*1}
\pgfmathsetmacro{\zbb}{\myvscale*3.95}
\pgfmathsetmacro{\zbc}{\myvscale*8.9}


\draw[teal,thick,dashed] (0,\zca) -- node[teal,above] {\Large{$C_1$}}
(1.95,\zca);
\draw[teal,thick,dashed] (0,\zcb) -- node[teal,above,xshift=1.3cm]
{\Large{$C_2$}} (3.9,\zcb);
\draw[teal,thick,dashed] (0,\zcc) -- node[teal,above,xshift=2.6cm]
{\Large{$C_3$}} (5.9,\zcc);


\draw[red,ultra thick] (0,0) node[blue,below] {\Large{$B_0$}} -- (0,7.7);
\draw[red,ultra thick] (1.95,\zba) node[blue,below] {\Large{$B_1$}} --
(1.95,\zca);
\draw[red,ultra thick] (3.9,\zbb) node[blue,below] {\Large{$B_2$}} --
(3.9,\zcb);
\draw[red,ultra thick] (5.92,\zbc) node[blue,below] {\Large{$B_3$}} --
(5.92,\zcc);


\draw[black,thick,-,smooth,domain=-0.915:0,samples=\nsamples,/pgf/fpu,
/pgf/fpu/output format=fixed] plot (\x,{\myvscale*pota(\x)});
\draw[black,thick,-,smooth,domain=0:1,samples=\nsamples,/pgf/fpu,/pgf/fpu/output
format=fixed] plot (\x, {\myvscale*potb(\x)});
\draw[black,thick,-,smooth,domain=1:2,samples=\nsamples,/pgf/fpu,/pgf/fpu/output
format=fixed] plot (\x, {\myvscale*potc(\x)});
\draw[black,thick,-,smooth,domain=2:3,samples=\nsamples,/pgf/fpu,/pgf/fpu/output
format=fixed] plot (\x, {\myvscale*potd(\x)});
\draw[black,thick,-,smooth,domain=3:4,samples=\nsamples,/pgf/fpu,/pgf/fpu/output
format=fixed] plot (\x, {\myvscale*pote(\x)});
\draw[black,thick,-,smooth,domain=4:5,samples=\nsamples,/pgf/fpu,/pgf/fpu/output
format=fixed] plot (\x, {\myvscale*potf(\x)});
\draw[black,thick,-,smooth,domain=5:6,samples=\nsamples,/pgf/fpu,/pgf/fpu/output
format=fixed] plot (\x, {\myvscale*potg(\x)});
\draw[black,thick,-,smooth,domain=6:6.75,samples=\nsamples,/pgf/fpu,
/pgf/fpu/output format=fixed] plot (\x, {\myvscale*poth(\x)});

\end{tikzpicture}

%% file: fig_transition_BC.tex
\begin{tikzpicture}[>=stealth',x=1cm,y=1cm,
declare function={pot(\x) = 0.5*(\x^6 - 6*\x^4 + 8*\x^2 + 0.5*\x);
stretch(\x) = 1.5*\x*sqrt(1+0.3*\x^2);
}]
\tikzset{mynode/.style={draw=white,shape=rectangle, minimum size=1cm}}


\def\octosite{2.2pt}

\node[mynode] (B0a) at (-3.7,-2.9) {$B_0$};
\pos{-3.7}{-2.9};
\iocto{black}{black}{black}{black}{white}{white}{white}{white}{0.5};

\node[mynode] (B0b) at (3.75,-2) {$B_0$};
\pos{3.75}{-2};
\iocto{white}{black}{black}{black}{white}{white}{black}{white}{0.5};

\node[mynode] (B1) at (0,-0.9) {$B_1$};
\pos{0}{-0.9};
\iocto{\oblue}{\ored}{\ored}{\ored}{\oblue}{\oblue}{\oblue}{\oblue}{0.5} ;

\node[mynode] (C1a) at (-1.5,2.2) {$C_1$};
\pos{-1.5}{2.2};
\iocto{\ooran}{\oviol}{\oviol}{\oviol}{\ogreen}{\ogreen}{\ogreen}{\ogreen}{0.5};

\node[mynode] (C1b) at (1.6,2.7) {$C_1$};
\pos{1.6}{2.7};
\iocto{\ogreen}{\oviol}{\oviol}{\oviol}{\ogreen}{\ogreen}{\ooran}{\ogreen}{0.5};

\draw[black,thick,-,smooth,domain=-2.18:2.15,samples=75,/pgf/fpu,
/pgf/fpu/output format=fixed] plot ({stretch(\x)}, {pot(\x)});

\end{tikzpicture}

%% file: fig_B0_N8_2.tex
\begin{tikzpicture}[>=stealth',x=1cm,y=1cm]
\tikzset{mynode/.style={draw=white,shape=rectangle, minimum size=1cm}}
\tikzset{bluenode/.style={thick,draw=black,fill=\oblue,shape=circle, inner
sep=0.08cm}}


\def\octosite{3pt}

\node[mynode] (8) at (0,0) {\Large $A_8$};
\pos{0}{0};
\iocto{black}{white}{black}{white}{black}{white}{black}{white}{0.75};

\node[mynode] (6a) at (3,-4) {\Large $A_6$};
\pos{3}{-4};
\iocto{white}{black}{black}{white}{black}{white}{black}{white}{0.75};

\node[mynode] (6b) at (5,-4) {\Large $A_6$};
\pos{5}{-4};
\iocto{white}{black}{black}{white}{black}{white}{white}{black}{0.75};

\node[mynode] (4a) at (8,0) {\Large $A_4$};
\pos{8}{0};
\iocto{white}{black}{black}{black}{white}{black}{white}{white}{0.75};

\node[mynode] (4b) at (10,0) {\Large $A_4$};
\pos{10}{0};
\iocto{white}{black}{black}{white}{black}{black}{white}{white}{0.75};

\node[mynode] (4') at (9,6) {\Large $A'_4$};
\pos{9}{6};
\iocto{white}{black}{black}{white}{white}{black}{black}{white}{0.75};

\node[mynode] (2) at (13,-4) {\Large $A_2$};
\pos{13}{-4};
\iocto{black}{black}{black}{black}{white}{white}{white}{white}{0.75};


\draw[thick,blue] (1.5,0) -- (4,0);
\draw[thick,blue,->] (4,0) -- node[blue,above] {\Large \rom{6}} (6.5,0);
\draw[thick,blue,->] (4,0) -- node[blue,right] {\Large \rom{5}} (4,-2.5);
\draw[thick,blue] (6.5,-4) -- (9,-4);
\draw[thick,blue,->] (9,-4) -- node[blue,below] {\Large \rom{6}} (11.5,-4);
\draw[thick,blue,->] (9,-4) -- node[blue,right] {\Large \rom{5}} (9,-1.5);
\draw[thick,blue] (9,1.5) -- node[blue,right,xshift=1mm] 
{\Large \rom{3}} (9,4.5);
\draw[thick,blue,->] (11.5,0) -- (13,0) -- node[blue,right] 
{\Large \rom{5}} (13,-2.5);


\node[bluenode] at (4,0) {};
\node[bluenode] at (9,-4) {};
\node[bluenode] at (9,3) {};
\node[bluenode] at (13,0) {};

\end{tikzpicture}

%% file: fig_N16.tex
\begin{tikzpicture}[>=stealth',shorten >=2pt,auto,node
distance=1.5cm, thick,main node/.style={circle,scale=0.7,minimum size=1.2cm,
fill=mymediumcyancolor!50,draw,font=\Large}]

\tikzset{bluenode/.style={thick,draw=black,fill=\oblue,shape=circle, inner
sep=0.06cm}}

  \node[main node] (16) at (0,0) {$A_{16}$};
  \node[main node] (12) at (3,0) {$A_{12}$};
  \node[main node] (8) at (6,0) {$A_{8}$};
  \node[main node] (4) at (9,0) {$A_{4}$};
  \node[main node] (14) at (1.5,-1.5) {$A_{14}$};
  \node[main node] (10) at (4.5,-1.5) {$A_{10}$};
  \node[main node] (6) at (7.5,-1.5) {$A_{6}$};
  \node[main node] (2) at (10.5,-1.5) {$A_{2}$};

  \node[main node] (8') at (6,3) {$A'_{8}$};
  \node[main node] (6') at (7.5,-4.5) {$A'_{6}$};
  \node[main node] (4') at (9,3) {$A'_{4}$};
  
  \node[bluenode] (121416) at (1.5,0) {};
  \node[bluenode] (81012) at (4.5,0) {};
  \node[bluenode] (468) at (7.5,0) {};
  \node[bluenode] (24) at (10.5,0) {};
  \node[bluenode] (101214) at (3,-1.5) {};
  \node[bluenode] (6810) at (6,-1.5) {};
  \node[bluenode] (246) at (9,-1.5) {};
  
  \node[bluenode] (88') at (6,1.5) {};
  \node[bluenode] (66') at (7.5,-3) {};
  \node[bluenode] (44') at (9,1.5) {};

  \draw[thick,blue] (121416) -- (16); 
  \draw[thick,blue,->] (121416) -- (14); 
  \draw[thick,blue,->] (121416) -- (12); 

  \draw[thick,blue] (101214) -- (14); 
  \draw[thick,blue,->] (101214) -- (12); 
  \draw[thick,blue,->] (101214) -- (10); 

  \draw[thick,blue] (81012) -- (12); 
  \draw[thick,blue,->] (81012) -- (10); 
  \draw[thick,blue,->] (81012) -- (8); 

  \draw[thick,blue] (6810) -- (10); 
  \draw[thick,blue,->] (6810) -- (8); 
  \draw[thick,blue,->] (6810) -- (6); 

  \draw[thick,blue] (468) -- (8); 
  \draw[thick,blue,->] (468) -- (6); 
  \draw[thick,blue,->] (468) -- (4); 

  \draw[thick,blue] (246) -- (6); 
  \draw[thick,blue,->] (246) -- (4); 
  \draw[thick,blue,->] (246) -- (2); 

  \draw[thick,blue] (24) -- (4); 
  \draw[thick,blue,->] (24) -- (2); 

  \draw[thick,blue] (88') -- (8); 
  \draw[thick,blue] (88') -- (8'); 

  \draw[thick,blue] (66') -- (6); 
  \draw[thick,blue] (66') -- (6'); 

  \draw[thick,blue] (44') -- (4); 
  \draw[thick,blue] (44') -- (4'); 

\end{tikzpicture}

%% file: fig_D2.tex
\begin{tikzpicture}
[>=stealth',x=18cm,y=6cm,
declare function={cubic(\x) = \x^3 - \x;
lambdamax(\x) = (sqrt(4/27)*(1-\x)*sqrt(1-\x)-\x)/(1-(6*sqrt(3)-9)*\x/4);
lambdacrit(\x) = sqrt(\x)*(1-\x);
}]

\pgfmathsetmacro{\lambdacrit}{sqrt(4/27)}


\draw[->,thick] (0.0,-\lambdacrit -0.1) -- (0.0,\lambdacrit +0.15);


\path[-,fill=blue!50,smooth,domain=0:0.25,samples=10,/pgf/fpu,
/pgf/fpu/output format=fixed] plot (\x, {lambdamax(\x)}) -- (0,0) --
(0,\lambdacrit);

\path[-,fill=blue!50,smooth,domain=0:0.25,samples=10,/pgf/fpu,
/pgf/fpu/output format=fixed] plot (\x, {-lambdamax(\x)}) -- (0,0) --
(0,-\lambdacrit);

\draw[very thick,blue,-,smooth,domain=0:0.25,samples=10,/pgf/fpu,
/pgf/fpu/output format=fixed] plot (\x, {lambdamax(\x)});

\draw[very thick,blue,-,smooth,domain=0.25:0,samples=10,/pgf/fpu,
/pgf/fpu/output format=fixed] plot (\x, {-lambdamax(\x)}) -- (0,\lambdacrit);


\newcommand*{\hgammamax}{0.08}

\path[-,fill=mycyancolor,smooth,domain=0:{\hgammamax},samples=20,/pgf/fpu,
/pgf/fpu/output format=fixed] plot (\x, {lambdacrit(\x)}) --
(\hgammamax,0) -- (0,0);

\path[-,fill=mycyancolor,smooth,domain=0:{\hgammamax},samples=20,/pgf/fpu,
/pgf/fpu/output format=fixed] plot (\x, {-lambdacrit(\x)}) -- 
(\hgammamax,0) -- (0,0);

\draw[thick,mydarkcyancolor,-,smooth,domain=0:0.03,samples=20,/pgf/fpu,
/pgf/fpu/output format=fixed] plot (\x, {lambdacrit(\x)});

\draw[thick,mydarkcyancolor,-,smooth,domain=0.03:{\hgammamax},
samples=10,/pgf/fpu, /pgf/fpu/output format=fixed] plot (\x, {lambdacrit(\x)});

\draw[thick,mydarkcyancolor,-,smooth,domain=0:0.03, samples=20,/pgf/fpu,
/pgf/fpu/output format=fixed] plot (\x, {-lambdacrit(\x)});

\draw[thick,mydarkcyancolor,-,smooth,domain=0.03:{\hgammamax},
samples=10,/pgf/fpu, /pgf/fpu/output format=fixed] plot (\x, {-lambdacrit(\x)});


\draw[->,thick] (-0.05,0.0) -- (0.35,0.0);


\draw[-,thick] (-0.01,\lambdacrit) -- (0.01,\lambdacrit);
\draw[-,thick] (-0.01,-\lambdacrit) -- (0.01,-\lambdacrit);
\draw[-,thick] (0.25,-0.02) -- (0.25,0.02);


\node[] at (0.32,-0.04) {$\gamma$}; 
\node[] at (0.03,\lambdacrit + 0.1) {$\lambda$}; 

\node[] at (0.25,-0.08) {$\tfrac14$}; 
\node[] at (-0.03,\lambdacrit) {$\lambdac$}; 
\node[] at (-0.045,-\lambdacrit) {$-\lambdac$}; 

\node[blue] at (0.16,0.23) {$D$};

\end{tikzpicture}

%% file: fig_horseshoe.tex
\begin{tikzpicture}
[>=stealth',x=2.5cm,y=2.5cm,
declare function={cubic(\x) = \x^3 - \x;
root(\la,\p,\k) = 
cos((rad(acos(\la*sqrt(27/\p^3)/2))+2*\k*pi)/3 r)*2*sqrt(\p/3);
}]

\newcommand*{\hlambda}{0.08}
\newcommand*{\hgamma}{0.19}
\pgfmathsetmacro{\xmax}{root(\hlambda,1,0)}
\pgfmathsetmacro{\xmin}{root(\hlambda,1,1)}

\draw[semithick] (\xmin,\xmin) -- (\xmin,\xmax) -- (\xmax,\xmax) --
(\xmax,\xmin) -- (\xmin,\xmin);

\pgfmathsetmacro{\hzz}{sqrt((1-\hgamma)/3)}

\pgfmathsetmacro{\hxaa}{\xmin}
\pgfmathsetmacro{\hxab}{root(\hlambda+\hgamma*(\xmin + \xmax)/2,1-\hgamma,1)}
\pgfmathsetmacro{\hxac}{root(\hlambda+\hgamma*\xmax,1-\hgamma,1)}

\path[-,fill=blue!50,smooth,domain={\hxaa}:{\hxab},samples=10,/pgf/fpu,
/pgf/fpu/output format=fixed] plot (\x, {2*\x - \xmin +
2*(cubic(\x)-\hlambda)/\hgamma}) -- (\hxab,\xmin) -- (\hxaa,\xmin);

\path[-,fill=blue!50,smooth,domain={\hxac}:{\hxab},samples=10,/pgf/fpu,
/pgf/fpu/output format=fixed] plot (\x, {2*\x - \xmax +
2*(cubic(\x)-\hlambda)/\hgamma}) -- (\hxaa,\xmin) -- (\hxab,\xmax) --
(\hxac,\xmax);

\draw[very thick,blue,-,smooth,domain={\hxaa}:{\hxab},samples=8,/pgf/fpu,
/pgf/fpu/output format=fixed] plot (\x, {2*\x - \xmin +
2*(cubic(\x)-\hlambda)/\hgamma}) -- (\hxab,\xmax) -- (\hxac,\xmax);

\draw[very thick,blue,-,smooth,domain={\hxac}:{\hxab},samples=8,/pgf/fpu,
/pgf/fpu/output format=fixed] plot (\x, {2*\x - \xmax +
2*(cubic(\x)-\hlambda)/\hgamma}) -- (\hxab,\xmin) -- (\hxaa,\xmin);

\pgfmathsetmacro{\hxba}{root(\hlambda+\hgamma*\xmin,1-\hgamma,2)}
\pgfmathsetmacro{\hxbb}{root(\hlambda+\hgamma*(\xmax+\xmin)/2,1-\hgamma,2)}
\pgfmathsetmacro{\hxbc}{root(\hlambda+\hgamma*\xmax,1-\hgamma,2)}

\path[-,fill=blue!50,smooth,domain={\hxba}:{\hxbb},samples=10,/pgf/fpu,
/pgf/fpu/output format=fixed] plot (\x, {2*\x - \xmin +
2*(cubic(\x)-\hlambda)/\hgamma}) -- (\hxbb,\xmin) -- (\hxba,\xmin);

\path[-,fill=blue!50,smooth,domain={\hxbc}:{\hxbb},samples=10,/pgf/fpu,
/pgf/fpu/output format=fixed] plot (\x, {2*\x - \xmax +
2*(cubic(\x)-\hlambda)/\hgamma}) -- (\hxba,\xmin) -- (\hxbb,\xmax) --
(\hxbc,\xmax);

\draw[very thick,blue,-,smooth,domain={\hxba}:{\hxbb},samples=5,/pgf/fpu,
/pgf/fpu/output format=fixed] plot (\x, {2*\x - \xmin +
2*(cubic(\x)-\hlambda)/\hgamma}) -- (\hxbb,\xmax) -- (\hxbc,\xmax);

\draw[very thick,blue,-,smooth,domain={\hxbc}:{\hxbb},samples=10,/pgf/fpu,
/pgf/fpu/output format=fixed] plot (\x, {2*\x - \xmax +
2*(cubic(\x)-\hlambda)/\hgamma}) -- (\hxbb,\xmin) -- (\hxba,\xmin);

\pgfmathsetmacro{\hxca}{root(\hlambda+\hgamma*\xmin,1-\hgamma,0)}
\pgfmathsetmacro{\hxcb}{root(\hlambda+\hgamma*(\xmax +\xmin)/2,1-\hgamma,0)}
\pgfmathsetmacro{\hxcc}{\xmax}

\path[-,fill=blue!50,smooth,domain={\hxca}:{\hxcb},samples=10,/pgf/fpu,
/pgf/fpu/output format=fixed] plot (\x, {2*\x - \xmin +
2*(cubic(\x)-\hlambda)/\hgamma}) -- (\hxcb,\xmin) -- (\hxca,\xmin);

\path[-,fill=blue!50,smooth,domain={\hxcc}:{\hxcb},samples=10,/pgf/fpu,
/pgf/fpu/output format=fixed] plot (\x, {2*\x - \xmax +
2*(cubic(\x)-\hlambda)/\hgamma}) -- (\hxca,\xmin) -- (\hxcb,\xmax) --
(\hxcc,\xmax);

\draw[very thick,blue,-,smooth,domain={\hxca}:{\hxcb},samples=10,/pgf/fpu,
/pgf/fpu/output format=fixed] plot (\x, {2*\x - \xmin +
2*(cubic(\x)-\hlambda)/\hgamma}) -- (\hxcb,\xmax) -- (\hxcc,\xmax);

\draw[very thick,blue,-,smooth,domain={\hxcc}:{\hxcb},samples=10,/pgf/fpu,
/pgf/fpu/output format=fixed] plot (\x, {2*\x - \xmax +
2*(cubic(\x)-\hlambda)/\hgamma}) -- (\hxcb,\xmin) -- (\hxca,\xmin);


\draw[->,thick] (-1.2,0) -- (1.3,0);
\draw[->,thick] (0,-1.2) -- (0,1.3);


\draw[thick] (-\hzz,-0.03) -- (-\hzz,0.03);
\draw[thick] (\hzz,-0.03) -- (\hzz,0.03);


\node[blue] at (\hxab +0.12,1.15) {$\cV_-$}; 
\node[blue] at (\hxbb -0.17,1.15) {$\cV_0$}; 
\node[blue] at (\hxcb +0.08,1.15) {$\cV_+$}; 

\node[] at (-\hzz,-0.12) {$-z_0$};
\node[] at (\hzz,-0.12) {$z_0$};

\node[] at (1.15,-0.12) {$x$};
\node[] at (0.12,1.15) {$y$};

\end{tikzpicture}

%% file: fig_horseshoe_H.tex
\begin{tikzpicture}
[>=stealth',x=2.5cm,y=2.5cm,
declare function={cubic(\x) = \x^3 - \x;
root(\la,\p,\k) = 
cos((rad(acos(\la*sqrt(27/\p^3)/2))+2*\k*pi)/3 r)*2*sqrt(\p/3);
}]

\newcommand*{\hlambda}{0.08}
\newcommand*{\hgamma}{0.19}
\pgfmathsetmacro{\xmax}{root(\hlambda,1,0)}
\pgfmathsetmacro{\xmin}{root(\hlambda,1,1)}

\draw[semithick] (\xmin,\xmin) -- (\xmin,\xmax) -- (\xmax,\xmax) --
(\xmax,\xmin) -- (\xmin,\xmin);

\pgfmathsetmacro{\hzz}{sqrt((1-\hgamma)/3)}

\pgfmathsetmacro{\hxaa}{\xmin}
\pgfmathsetmacro{\hxab}{root(\hlambda+\hgamma*(\xmin + \xmax)/2,1-\hgamma,1)}
\pgfmathsetmacro{\hxac}{root(\hlambda+\hgamma*\xmax,1-\hgamma,1)}

\path[-,fill=green!70!black,smooth,domain={\hxaa}:{\hxab},samples=10,/pgf/fpu,
/pgf/fpu/output format=fixed] plot ({2*\x - \xmin +
2*(cubic(\x)-\hlambda)/\hgamma}, \x) -- (\hxab,\xmin) -- (\hxaa,\xmin);

\path[-,fill=green!70!black,smooth,domain={\hxac}:{\hxab},samples=10,/pgf/fpu,
/pgf/fpu/output format=fixed] plot ({2*\x - \xmax +
2*(cubic(\x)-\hlambda)/\hgamma}, \x) -- (\xmin,\hxaa) -- (\xmax,\hxab) --
(\xmax,\hxac);

\draw[very
thick,green!30!black,-,smooth,domain={\hxaa}:{\hxab},samples=10,/pgf/fpu,
/pgf/fpu/output format=fixed] plot ({2*\x - \xmin +
2*(cubic(\x)-\hlambda)/\hgamma}, \x) -- (\xmax,\hxab) -- (\xmax,\hxac);

\draw[very
thick,green!30!black,-,smooth,domain={\hxac}:{\hxab},samples=8,/pgf/fpu,
/pgf/fpu/output format=fixed] plot ({2*\x - \xmax +
2*(cubic(\x)-\hlambda)/\hgamma}, \x) -- (\xmin,\hxab) -- (\xmin,\hxaa);

\pgfmathsetmacro{\hxba}{root(\hlambda+\hgamma*\xmin,1-\hgamma,2)}
\pgfmathsetmacro{\hxbb}{root(\hlambda+\hgamma*(\xmax+\xmin)/2,1-\hgamma,2)}
\pgfmathsetmacro{\hxbc}{root(\hlambda+\hgamma*\xmax,1-\hgamma,2)}

\path[-,fill=green!70!black,smooth,domain={\hxba}:{\hxbb},samples=10,/pgf/fpu,
/pgf/fpu/output format=fixed] plot ({2*\x - \xmin +
2*(cubic(\x)-\hlambda)/\hgamma}, \x) -- (\xmin,\hxbb) -- (\xmin,\hxba);

\path[-,fill=green!70!black,smooth,domain={\hxbc}:{\hxbb},samples=10,/pgf/fpu,
/pgf/fpu/output format=fixed] plot ({2*\x - \xmax +
2*(cubic(\x)-\hlambda)/\hgamma}, \x) -- (\xmin,\hxba) -- (\xmax,\hxbb) --
(\xmax,\hxbc);

\draw[very
thick,green!30!black,-,smooth,domain={\hxba}:{\hxbb},samples=5,/pgf/fpu,
/pgf/fpu/output format=fixed] plot ({2*\x - \xmin +
2*(cubic(\x)-\hlambda)/\hgamma}, \x) -- (\xmax,\hxbb) -- (\xmax,\hxbc);

\draw[very
thick,green!30!black,-,smooth,domain={\hxbc}:{\hxbb},samples=10,/pgf/fpu,
/pgf/fpu/output format=fixed] plot ({2*\x - \xmax +
2*(cubic(\x)-\hlambda)/\hgamma}, \x) -- (\xmin,\hxbb) -- (\xmin,\hxba);

\pgfmathsetmacro{\hxca}{root(\hlambda+\hgamma*\xmin,1-\hgamma,0)}
\pgfmathsetmacro{\hxcb}{root(\hlambda+\hgamma*(\xmax +\xmin)/2,1-\hgamma,0)}
\pgfmathsetmacro{\hxcc}{\xmax}

\path[-,fill=green!70!black,smooth,domain={\hxca}:{\hxcb},samples=10,/pgf/fpu,
/pgf/fpu/output format=fixed] plot ({2*\x - \xmin +
2*(cubic(\x)-\hlambda)/\hgamma}, \x) -- (\xmin,\hxcb) -- (\xmin,\hxca);

\path[-,fill=green!70!black,smooth,domain={\hxcc}:{\hxcb},samples=10,/pgf/fpu,
/pgf/fpu/output format=fixed] plot ({2*\x - \xmax +
2*(cubic(\x)-\hlambda)/\hgamma}, \x) -- (\xmin,\hxca) -- (\xmax,\hxcb) --
(\xmax,\hxcc);

\draw[very
thick,green!30!black,-,smooth,domain={\hxca}:{\hxcb},samples=10,/pgf/fpu,
/pgf/fpu/output format=fixed] plot ({2*\x - \xmin +
2*(cubic(\x)-\hlambda)/\hgamma}, \x) -- (\xmax,\hxcb) -- (\xmax,\hxcc);

\draw[very
thick,green!30!black,-,smooth,domain={\hxcc}:{\hxcb},samples=10,/pgf/fpu,
/pgf/fpu/output format=fixed] plot ({2*\x - \xmax +
2*(cubic(\x)-\hlambda)/\hgamma}, \x) -- (\xmin,\hxcb) -- (\xmin,\hxca);


\draw[->,thick] (-1.2,0) -- (1.3,0);
\draw[->,thick] (0,-1.2) -- (0,1.3);


\draw[thick] (-0.03,-\hzz) -- (0.03,-\hzz);
\draw[thick] (-0.03,\hzz) -- (0.03,\hzz);


\node[green!30!black] at (-1.12,-\hxcb +0.03) {$\cH_-$}; 
\node[green!30!black] at (-1.12,-\hxbb ) {$\cH_0$}; 
\node[green!30!black] at (-1.12,-\hxab +0.03) {$\cH_+$}; 

\node[] at (-0.2,-\hzz) {$-z_0$};
\node[] at (-0.14,\hzz) {$z_0$};

\node[] at (1.15,-0.12) {$x$};
\node[] at (0.12,1.15) {$y$};

\end{tikzpicture}